%% file: main_modified.tex
\documentclass[journal,twoside,web]{ieeecolor}

\usepackage{mathptmx} 
\usepackage{amssymb}

\usepackage{amsmath, amsthm}
\usepackage{graphicx}
\usepackage{subcaption}
\usepackage{graphicx}
\usepackage{cite}
\usepackage{siunitx}
\usepackage{algorithmic}
\usepackage{algorithm}
\usepackage{booktabs}

\usepackage{tikz}
\usepackage{graphicx}

\newcommand{\bmat}[1]{\begin{bmatrix}#1\end{bmatrix}}
\newcommand{\tp}{\mathsf{T}}
\newcommand{\norm}[1]{\lVert{#1}\rVert}
\newcommand{\mM}{{\mathsf{M}}}
\newcommand{\R}{\mathbb{R}}

\newtheorem{theorem}{Theorem}
\newtheorem{lemma}{Lemma}

\newtheorem{definition}{Definition} 
\newtheorem{proposition}{Proposition}
\newtheorem{remark}{Remark}

\newtheorem{assumption}{Assumption}

\newcommand{\bmtx}{\begin{bmatrix}}
\newcommand{\emtx}{\end{bmatrix}}
\newcommand{\bsmtx}{\left[ \begin{smallmatrix}} 
\newcommand{\esmtx}{\end{smallmatrix} \right]}

\usepackage{generic}
\usepackage{cite}
\usepackage{amsfonts}
\usepackage{algorithmic}
\usepackage{graphicx}
\usepackage{algorithm,algorithmic}
\usepackage{hyperref}
\hypersetup{hidelinks=true}
\usepackage{textcomp}

\def\BibTeX{{\rm B\kern-.05em{\sc i\kern-.025em b}\kern-.08em
    T\kern-.1667em\lower.7ex\hbox{E}\kern-.125emX}}
\begin{document}

\title{
Scalable Incremental Robustness Analysis of Neural Network Feedback Systems 
}

\author{
Zichen Wang$^{1}$, Peter Seiler$^{2}$, \IEEEmembership{Fellow, IEEE},  Geir Dullerud$^3$, \IEEEmembership{Fellow, IEEE}, and Bin Hu$^{1}$, \IEEEmembership{Member, IEEE}
\thanks{\textsuperscript{\dag}Generative AI tools were used solely to improve
the language and presentation of the manuscript. All scientific ideas,
analyses, results, and conclusions are those of the authors, who take full
responsibility for the content of this work.}
\thanks{*This work was supported by  the
AFOSR award FA9550-23-1-0732.}% <-this % stops a space
\thanks{$^{1}$Zichen Wang and Bin Hu are with the Coordinated Science Laboratory
and the Department of Electrical and Computer Engineering, University
of Illinois Urbana–Champaign
        ({\tt\small zichenw6@illinois.edu; binhu7@illinois.edu}).
$^{2}$Peter Seiler is with the Department of Electrical Engineering and Computer Science, at the University of Michigan
({\tt\small pseiler@umich.edu}). 
$^{3}$Geir Dullerud is with the Electrical and Computer Engineering Department at the University of Minnesota
        ({\tt\small dullerud@umn.edu}).}
} 
 
\maketitle

\begin{abstract}
Semidefinite programming (SDP) certificates for feedback systems containing deep neural networks (NNs) typically scale with the total number of neurons, whereas small-gain tests are scalable but can be highly conservative. This paper develops a unified and scalable framework for incremental robust stability and performance analysis of feedback interconnections involving high-dimensional NNs and unmodeled dynamics. By combining a structured decomposition of the full-order SDP condition with scalable Lipschitz constant estimation algorithms, we derive reduced verification conditions that certify incremental convergence and incremental $\ell_2$-gain bounds. The dimensions of the resulting control-analysis linear matrix inequalities (LMIs) depend only on the widths of the last two network layers and are \textit{independent of  network depth}. The framework preserves the coupling between the plant and the NN, with the incremental small-gain condition recovered as a special case. To further reduce conservatism, we develop a multi-round alternating update scheme that iteratively refines the coupling variables while preserving scalability. Numerical experiments show that the proposed framework achieves  state-of-the-art incremental $\ell_2$-gain bounds for large-scale NN feedback systems.

\end{abstract}


\input{Intro_modified}

\input{Problem_formulation_modified}

\input{IQC}

\input{Scalable}

\input{Normal_stability_verification_modified}

\input{Further_properties}

\input{Numerical_modified}

\section{Conclusion}
\label{sec:conclud}

This paper develops a unified and scalable framework for incremental
stability and performance analysis of feedback interconnections involving
NNs and additional uncertainties. By combining a structured decomposition
of the full-order SDP condition with scalable Lipschitz constant estimation,
the resulting control-analysis LMIs depend only on the widths
of the last two network layers.
A multi-round alternating update scheme is further
introduced to reduce conservatism. The proposed conditions guarantee
incremental convergence and a finite upper bound on the incremental
$\ell_2$ gain, and, under the additional assumptions of Theorem~\ref{Theorem:bounded}, also
guarantee bounded individual state trajectories. Pointwise asymptotic
convergence to a common limit may require an additional assumption.
Numerical experiments demonstrate the scalability and reduced conservatism
of the proposed framework compared with existing scalable approaches.

\input{Appendix_modified}

\section*{References}

\bibliographystyle{IEEEtran}
\bibliography{main} 

\begin{IEEEbiography}[{\includegraphics[width=1in,height=1.25in,clip,keepaspectratio]{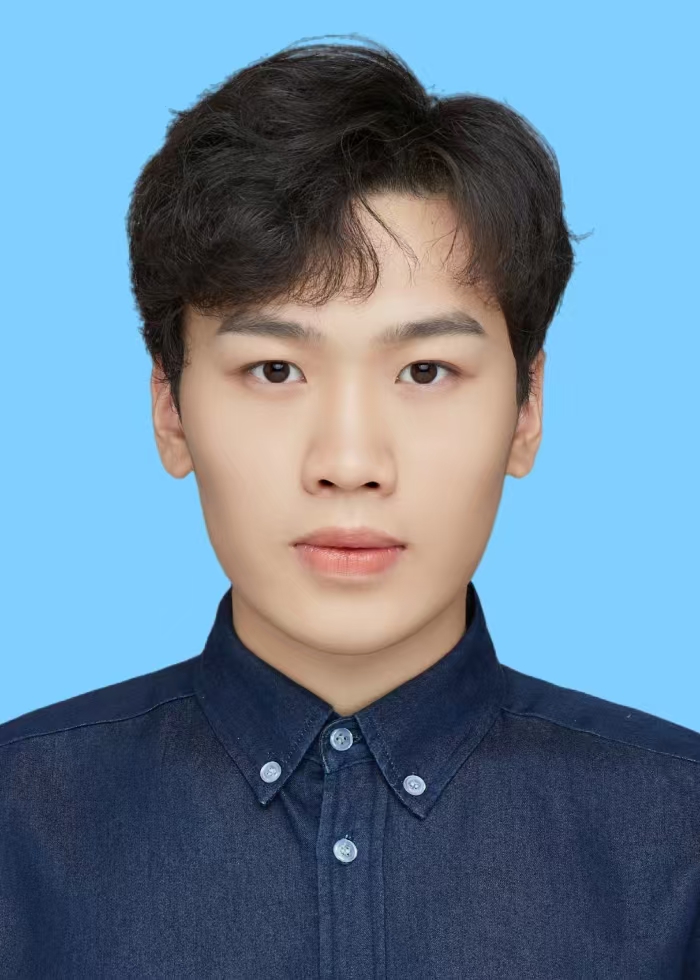}}]{Zichen Wang} received the B.E. degree in Electronic Information Engineering from Southwest University, Chongqing, China, in 2024. He is currently a Ph.D. student in the Department of Electrical and Computer Engineering and the Coordinated Science Laboratory at the University of Illinois Urbana-Champaign.

His research interests include analysis and verification of learning-based control systems and reinforcement learning.
\end{IEEEbiography}

\begin{IEEEbiography}[{\includegraphics[width=1in,height=1.25in,clip,keepaspectratio]{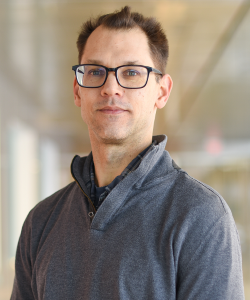}}]{Peter
Seiler}
  (Fellow, IEEE) received B.S. degrees in Mathematics and Mechanical
  Engineering from the University of Illinois, Urbana-Champaign in
  1996.  He received the Ph.D. degree in Mechanical Engineering from
  the University of California, Berkeley, in 2001. He is currently a
Professor in Electrical Engineering and Computer Science at the
University of Michigan, Ann Arbor. He was previously
  on the faculty at the University of Minnesota, Twin Cities, in
  Aerospace Engineering and Mechanics from 2011-2019.  He was also a
  Principal Scientist R\&D at the Honeywell Labs in Minneapolis
  Minnesota from 2004-2008.  His current research interests include
  merging robust control techniques with online optimization and
  learning-based methods. Dr. Seiler received the National Science
  Foundation CAREER Award in 2013, Brockett-Willems Outstanding Paper
  Award in 2021, and the O. Hugo Schuck Best Paper Award in 2003.
\end{IEEEbiography}

\begin{IEEEbiography}[{\includegraphics[width=1in,height=1.25in,clip,keepaspectratio]{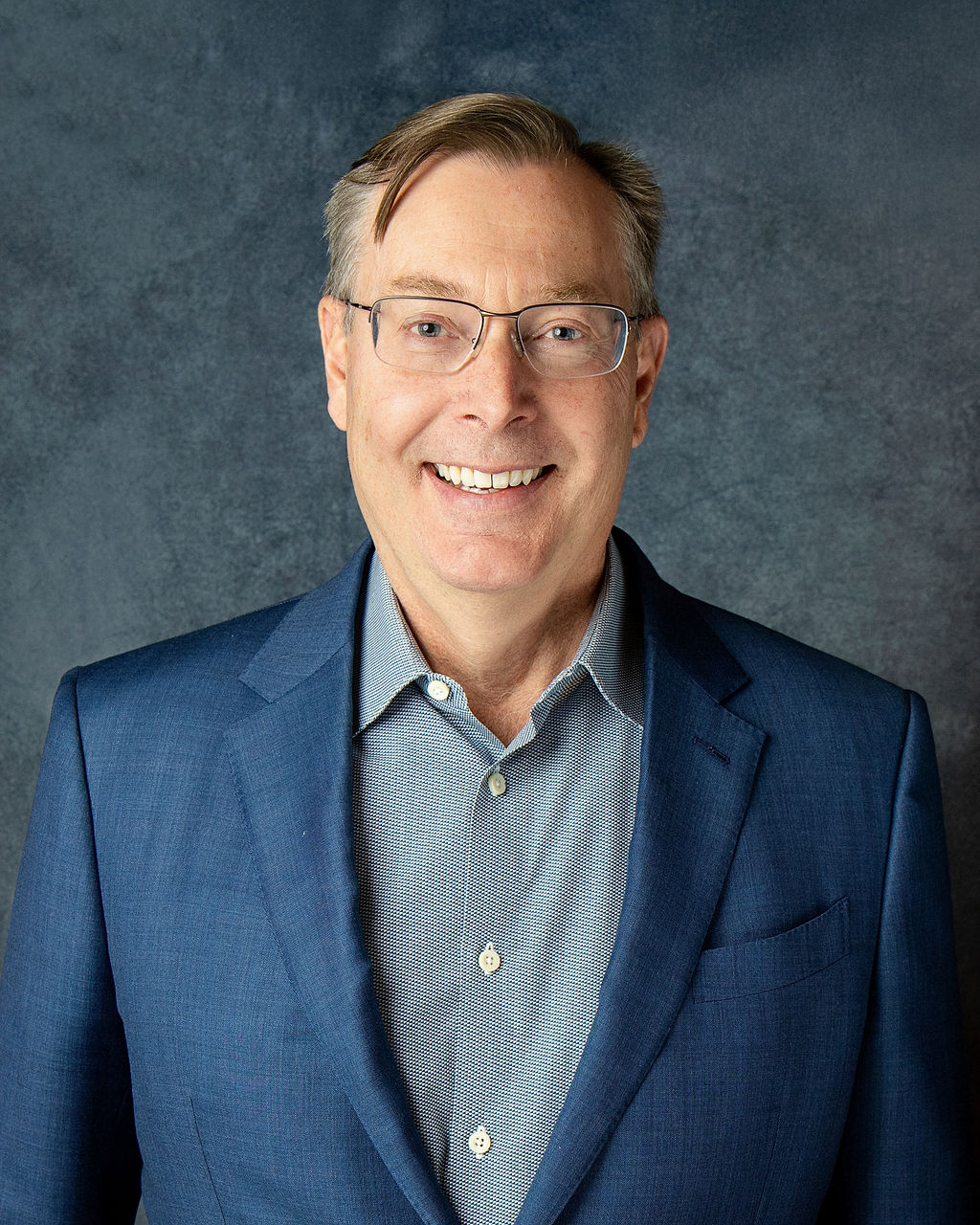}}]{Geir E. Dullerud} (Fellow, IEEE) 
is Department Head and Professor of Electrical and Computer Engineering at the University of Minnesota.  Prior to this, 1998-2025, he was on the faculty of the University of Illinois as a member of the Mechanical Science and Engineering Department, and the Coordinated Science Laboratory.  Preceding this he was an Assistant Professor in Applied Mathematics  with the University of Waterloo, Waterloo, ON, Canada, before which he held the appointment of Research Fellow and Lecturer with the Department of Electrical Engineering, California Institute of Technology, Pasadena, CA, USA.
He has held visiting positions  with the KTH Royal Institute of Technology, Stockholm, Sweden, and  Stanford University, Stanford, CA, USA.  He holds a PhD from Cambridge University in engineering.  His current research interests include AI and machine learning in control, cooperative robotics, networks, and system verification. 
He is a Senior Editor for the IEEE Transactions on Automatic Control.  

\end{IEEEbiography}

\begin{IEEEbiography}[{\includegraphics[width=1in,height=1.25in,clip,keepaspectratio]{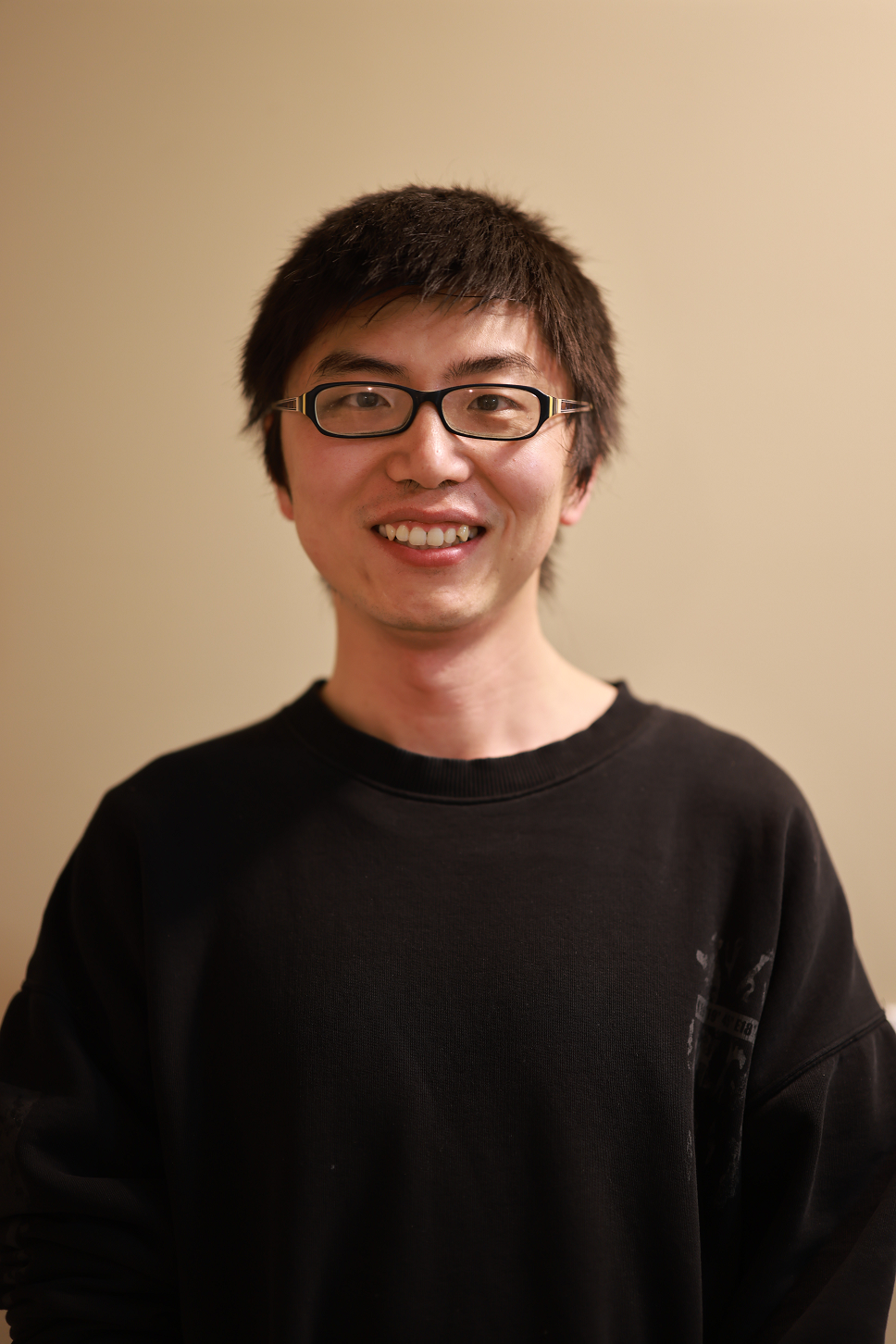}}]{Bin Hu}
  (Member, IEEE)  received the B.S. degree in theoretical and applied mechanics from
the University of Science and Technology of
China, Hefei, China, in 2008, the M.S. degree in
computational mechanics from Carnegie Mellon
University, Pittsburgh, PA, USA, in 2010, and
the Ph.D. degree in aerospace engineering and
mechanics from the University of Minnesota,
Minneapolis, MN, USA, in 2016.
Between July 2016 and July 2018, He was
a Postdoctoral Researcher with the Wisconsin
Institute for Discovery, University of Wisconsin-Madison, Madison, WI,
USA. He is currently an Associate Professor with the Department of
Electrical and Computer Engineering, University of Illinois Urbana
Champaign, Champaign, IL, USA, and holds affiliations with the
Coordinated Science Laboratory and the Siebel School of Computing and Data Science. His research focuses on establishing
fundamental connections between control and machine learning.  Bin received the NSF CAREER award and the
Amazon Research Award in 2021, and the O. Hugo Schuck Best Paper
Award in 2024
\end{IEEEbiography}

\end{document}

%% file: Intro_modified.tex
\section{Introduction}

\IEEEPARstart{N}{eural} networks are increasingly incorporated into feedback control systems as learned controllers, policy representations, and nonlinear decision-making components in applications such as robotics, autonomous systems, and reinforcement learning \cite{jiang2020learning,bucsoniu2018reinforcement}. In many NN-in-the-loop architectures, a feedforward NN can be viewed as a static nonlinear mapping interconnected with a dynamical system. Rigorous guarantees on closed-loop stability and performance are therefore important for reliable deployment. Such guarantees, however, become increasingly difficult to compute as modern NNs grow in depth and width. In addition, analysis relative to a prescribed equilibrium can be restrictive in the presence of NN biases, disturbances, or uncertain components. This motivates an incremental viewpoint, which studies the behavior of pairs of trajectories and provides a natural framework for characterizing convergence and input--output sensitivity \cite{angeli2002lyapunov,lohmiller1998contraction,9029867,hashemi2021incremental}.

Quadratic constraints, integral quadratic constraints (IQCs), and dissipativity provide a systematic framework for analyzing nonlinear and uncertain feedback interconnections \cite{megretski1997system,willems2007dissipative,seiler13,hu2017robustness}. These tools have been applied to NN feedback systems by exploiting quadratic constraints satisfied by slope-restricted activation functions \cite{yin2021stability,Sahel2024,richardson2023strengthened,hashemi2021incremental,desouza2023event,zhang2025robust}. By retaining the layer-wise structure of the NN, the resulting SDP formulations preserve detailed information about the interaction between the NN nonlinearity and the surrounding dynamics \cite{yin2021stability,richardson2023strengthened,hashemi2021incremental}. Such formulations can yield expressive certificates for closed-loop stability, robustness, and, in some settings, input--output performance \cite{Sahel2024,desouza2023event}. The main limitation of such full-order formulations is scalability: the dimensions of their semidefinite constraints typically grow with the total number of neurons, leading to rapidly increasing computational and memory requirements for deep and wide networks.

A natural scalable alternative is to summarize the NN by a global Lipschitz bound and combine it with a small-gain argument. Significant progress has been made in scalable Lipschitz constant estimation for deep NNs through sparsity exploitation, structured decompositions, first-order optimization, and compositional constructions \cite{xue2022chordal,xue2024chordal,wang2024scalability,xueclipse,syed2025improved}. When such estimates are used in small-gain-based closed-loop analysis,
the resulting verification conditions can have dimensions that are
independent of the internal depth of the NN \cite{Wang2025ScalableDF},
making them applicable to very deep networks. This scalability, however, comes at the cost of conservatism: replacing the detailed NN description by a single global gain discards structural information about the coupling between the plant dynamics and the NN. Existing methods therefore exhibit a fundamental tradeoff between expressive full-order SDP certificates that preserve detailed plant--NN coupling but scale poorly, and highly scalable small-gain certificates that discard much of this coupling and may be substantially more conservative.

This tradeoff becomes more pronounced when the NN is only one component of a larger uncertain feedback interconnection. Practical systems may additionally contain actuator uncertainty, modeling error, unmodeled dynamics, or other uncertain nonlinear or dynamical components \cite{zhou1996robust,skogestad2005multivariable}. IQC-based descriptions provide a natural mechanism for incorporating uncertain and nonlinear components into a common stability and performance analysis framework. Related work has analyzed NN feedback systems with additional
perturbations using quadratic constraints and IQCs
\cite{yin2021stability,desouza2023event}, considered robust stability
under structured model uncertainty \cite{zhang2025robust}, and used
IQCs and dissipativity-based formulations to certify closed-loop
robustness and performance \cite{junnarkar2026dissipativity}.
 However, directly combining detailed NN descriptions with additional unmodeled dynamics generally retains the scalability limitations of full-order SDP formulations, whereas collapsing the NN to a global gain can sacrifice useful structural coupling.

These observations motivate a scalable verification framework that preserves informative plant--NN coupling, accommodates additional IQC-described uncertainties, and avoids the conservatism of a purely gain-based reduction. To address this gap, this paper develops a unified framework for incremental stability and performance analysis of uncertain NN feedback interconnections. The main contributions are summarized as follows:

\begin{itemize}
\item We develop a scalable verification framework for feedback interconnections involving high-dimensional NNs and uncertainties, enabling a systematic tradeoff between computational and memory efficiency and the conservatism of the resulting stability and performance certificates.

\item The framework provides sufficient conditions for incremental
convergence and finite incremental $\ell_2$ gain and, under
additional assumptions, for boundedness of closed-loop trajectories.

\item We show that, under an additional assumption ensuring the existence of a convergent trajectory, incremental convergence implies global convergence to a common limit point. We further provide a counterexample demonstrating that such pointwise convergence cannot, in general, be guaranteed without this assumption.
\end{itemize}

The remainder of this paper is organized as follows.
Section~\ref{sec:preliminiries} introduces the problem formulation.
Section~\ref{sec:background} develops the full-order incremental IQC
framework, including incremental IQCs, dissipation-based analysis,
and the incremental small-gain method.
Section~\ref{scalablesection} reviews scalable Lipschitz constant estimation
methods for deep NNs.
Section~\ref{sec:normal} develops the proposed scalable verification framework
and an alternating update scheme for reducing conservatism.
Section~\ref{sec:sufficient_condition} discusses further implications and
limitations of the proposed sufficient conditions.
Finally, Section~\ref{sec:num} presents the numerical results.

%% file: Problem_Formulation_modified.tex
\section{Problem Formulation}
\label{sec:preliminiries}

\vspace{0.1in}

% \begin{figure}[t!]
% \centering
% \scalebox{0.9}{
% \begin{picture}(180,120)(20,-20)
%  \thicklines
%  \put(80,25){\framebox(30,30){$G$}}
% % I/O for Delta
% % \put(80,75){\dashbox(30,30){$\Delta$}}
%  \put(80,70){\framebox(30,30){$\Delta_{\text{NN}}$}}
%   \put(80,-20){\framebox(30,30){$\Delta_{U}$}}
%  \put(42,64){$v$}
%  \put(42,9){$p$}
%  \put(55,50){\line(1,0){25}}  
%   \put(80,40){\vector(-1,0){40}} 
% \put(55,50){\line(0,1){35}}
% \put(55,30){\line(0,-1){35}}
% \put(55,-5){\vector(1,0){25}}
% \put(110,-5){\line(1,0){25}}
% \put(135,-5){\line(0,1){35}}
%  \put(55,30){\line(1,0){25}} 
%  \put(55,85){\vector(1,0){25}}  
%  \put(143,64){$w$}
%  \put(143,9){$q$}
%  \put(135,85){\line(-1,0){25}}  
%  \put(135,50){\line(0,1){35}} 
%   \put(150,40){\vector(-1,0){40}} 
%   \put(155,38){$d$} 
%   \put(30,38){$e$} 
%  \put(135,50){\vector(-1,0){25}}  
%   \put(135,30){\vector(-1,0){25}}
% \end{picture}
% } % End scalebox
% \caption{Block diagram of the feedback interconnection $F(G,\Delta_{\mathrm{NN}},\Delta_U)$ with exogenous input $d$ and output $e$}
% %\label{fig:fdbd-uncertainty}
% \end{figure}

\begin{figure}[h!t]
\centering
\begin{picture}(110,90)(40,20)
 \thicklines
 \put(75,25){\framebox(40,40){$G$}}
 \put(143,40){$d$}
 \put(150,35){\vector(-1,0){35}}  
 \put(42,40){$e$}
 \put(75,35){\vector(-1,0){35}}  
 \put(75,25){\framebox(40,40){$G$}}
 % \put(76,71){\framebox(38,38){
 %    $\begin{matrix}
 %    \Delta_\text{NN} &  \\
 %      & \Delta_U
 %    \end{matrix}$
 %    }}
 \put(76,71){\framebox(38,38){}}
 \put(78,95){$\Delta_{\text{NN}}$}
 \put(100,80){$\Delta_{\text{U}}$}
 \put(35,75){$\bmtx v \\ p \emtx$}
 \put(55,55){\line(1,0){20}}  
 \put(55,55){\line(0,1){35}}  
 \put(55,90){\vector(1,0){21}}  
 \put(138,75){$\bmtx w \\ q \emtx$}
 \put(135,90){\line(-1,0){21}}  
 \put(135,55){\line(0,1){35}}  
 \put(135,55){\vector(-1,0){20}}  
\end{picture}
\caption{Block diagram of the feedback interconnection $F(G,\Delta_{\mathrm{NN}},\Delta_\text{U})$ with exogenous input $d$ and performance output $e$.}
\label{fig:fdbd-uncertainty}
\end{figure}
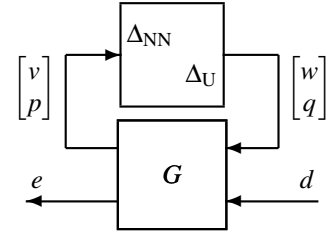

\subsection{Notation}
Let $\mathbb{Z}_{\ge 0}$ denote the set of non-negative integers, $[N] := \{1,\dots,N\}$, and $\mathbb{R}^n$ the set of real vectors of dimension $n$. The Euclidean (or $2$-) norm for a vector $x\in \R^n$ is defined by $\|x\|_2:= \sqrt{ x^\tp x }$. Let $\ell^n_{2e}$ denote the set of sequences $v = \{v_k\}_{k=0}^\infty$, where $v_k \in \mathbb{R}^n$.  
The $\ell_2$-norm of a sequence $v \in \ell^n_{2e}$ is defined as $\|v\|_2 := \left( \sum_{k=0}^{\infty} \|v_k\|_2^2 \right)^{1/2}$. 
The space $\ell_{2}^{n}$ consists of sequences $v \in \ell_{2e}^n$ such that $\|v\|_2 < \infty$. 
 Next, $\mathbb{R}^{n\times m}$ denotes the set of $n\times m$ real matrices. 
The $n\times n$ identity matrix is denoted by $I_n$, while $0_n$ and $0_{n\times m}$ denote the $n\times n$ and $n\times m$ zero matrices, respectively. The maximum singular value (spectral norm) for any matrix $A\in \R^{n\times m}$ is denoted by $\|A\|_2$.  The eigenvalues of  a square symmetric matrix $A\in \R^{n\times n}$ are real and $\lambda_{\min}(A)$ denotes the smallest eigenvalue. The operator $\operatorname{diag}(A,B)$ denotes the block-diagonal matrix with matrices $A$ and $B$ on its diagonal.  Finally, $\mathbb{RH}_\infty$ denotes the set of real-rational, stable, and proper transfer functions.

\subsection{System Model and Objectives}

Consider the uncertain discrete-time feedback interconnection $F(G,\Delta_{\mathrm{NN}},\Delta_\text{U})$ in Fig. \ref{fig:fdbd-uncertainty}, consisting of an LTI plant $G$, a NN nonlinearity $\Delta_{\mathrm{NN}}$, and an uncertain perturbation $\Delta_\text{U}$.
The dynamics of $G$ are described by
\begin{align}\label{eq:uncertain_sys}
\begin{split}
x_{k+1} &= A_G \, x_k \,\, + B_{G1} \,  w_k
  \,\,\,\,  + B_{G2} \, q_k 
  \,\,\,\,\,\, + B_{G3} \,  d_k \\ 
v_k &= C_{G1} \,  x_k + \phantom{{}D_{G11}w_k +}   D_{G12} \,  q_k + \, D_{G13}\,  d_k\\
p_k &= C_{G2} \,  x_k + D_{G21} \, w_k  + D_{G22} q_k + D_{G23}\,  d_k \\
e_k &= C_{G3} \,  x_k + D_{G31}\,  w_k + D_{G32}\,  q_k + D_{G33}\,  d_k\\
\end{split}
 \end{align}
for all $k \in \mathbb{Z}_{\ge 0}$,
where $x_k \in \mathbb{R}^{n_x}$ is the state, 
$(p_k,q_k) \in \mathbb{R}^{n_p} \times \mathbb{R}^{n_q}$ denote the input-output pair of $\Delta_\text{U}$, 
$(v_k,w_k) \in \mathbb{R}^{n_v} \times \mathbb{R}^{n_w}$ denote the input-output pair of $\Delta_{\mathrm{NN}}$, 
$d_k \in \mathbb{R}^{n_d}$ is the
exogenous input, and 
$e_k \in \mathbb{R}^{n_e}$ is the performance output. We next state a standing assumption on the perturbation
$\Delta_\text{U}$.

\begin{assumption}\label{assumption1}
For every admissible perturbation $\Delta_{\operatorname{U}}:\ell_{2e}^{n_p}\rightarrow\ell_{2e}^{n_q}$, the feedback interconnection $F(G,\Delta_{\rm NN},\Delta_{\operatorname{U}})$ is well posed, i.e., for every admissible initial condition and exogenous input, the feedback equations admit a unique trajectory. Moreover, $\Delta_{\operatorname{U}}$ is causal and satisfies $\Delta_{\operatorname{U}}(0)=0$.
\end{assumption}

The perturbation can represent a broad class of uncertainties \cite{megretski1997system}, including actuator saturation and other slope-restricted nonlinearities, as well as time-delay effects and unmodeled dynamics.

The NN block $\Delta_{\mathrm{NN}}$ is a feedforward network with $N$  layers. 
The input–output relation of $\Delta_{\mathrm{NN}}$ is given by
\begin{align}\label{eq:Delta}
\begin{split}
h_k^0 &= v_k,\\
h_k^l &= \phi\!\left(W_l h_k^{l-1} + b_l\right), \quad l \in [N],\\
w_k &= h_k^N .
\end{split}
\end{align}
For each layer $l \in [N]$, $W_l \in \mathbb{R}^{n_l \times n_{l-1}}$ and $b_l \in \mathbb{R}^{n_l}$ denote the weight matrix and bias vector. 
The layer dimensions satisfy $n_0 = n_v$ and $n_N = n_w$, and the total number of neurons is $n_\phi := \sum_{l=1}^N n_l$.
In the NN architecture considered here, the output of the last nonlinear
layer is directly taken as the NN output, i.e., $w_k=h_k^N$, and no
additional output layer is included. If an additional linear output layer
with weight matrix $W_{N+1}$ is present, its weight can, without loss of
generality, be absorbed into $B_{G1}$, $D_{G21}$, and $D_{G31}$.
An output bias is not considered. The activation function $\phi$ acts componentwise on each layer and is
assumed to be slope-restricted in $[0,1]$ and to satisfy $\phi(0)=0$.
These assumptions are satisfied by several standard activation functions,
including ReLU, tanh, and leaky ReLU with a leak coefficient in $[0,1]$.

When the bias terms $\{b_l\}_{l=1}^N$ and the external input $d$ are zero,
the origin is an equilibrium point of \eqref{eq:uncertain_sys}.
In the presence of nonzero biases, however, the equilibrium may shift or may
not exist, making stability analysis with respect to a fixed operating point
inapplicable.  
We therefore adopt an incremental viewpoint and study pairs of system
trajectories. Our objective is to characterize both trajectory convergence
and sensitivity to input perturbations. The latter is quantified via the
notion of incremental $\ell_2$ gain, which captures the input–output behavior
between pairs of trajectories and does not rely on a specific equilibrium.
We next define the incremental $\ell_2$ gain.

\begin{definition}[Incremental $\ell_2$ gain \cite{9029867}]\label{def:incrementalL2gain}
The system \eqref{eq:uncertain_sys} is said to have an incremental $\ell_2$-gain bound $\sqrt{\gamma}$ if, for any input pair $d,d'\in\ell_2^{n_d}$ and any initial conditions $x_0, x_0' \in \mathbb{R}^{n_x}$, the corresponding trajectories satisfy, for all $K \in \mathbb{Z}_{\ge 0}$,
\[
\sum_{k=0}^K \|e_k - e_k'\|_2^2 \le \gamma \sum_{k=0}^K \|d_k - d_k'\|_2^2 + \chi(x_0, x_0'),
\]
where $\chi : \mathbb{R}^{n_x} \times \mathbb{R}^{n_x} \to \mathbb{R}_{\ge 0}$ is a nonnegative function satisfying $\chi(x,x)=0$.
\end{definition}

Our objective is to derive scalable and verifiable sufficient conditions under which:
\begin{itemize}
    \item the closed-loop system is incrementally convergent and its state trajectories remain bounded; and
    \item the interconnection $F(G,\Delta_{\mathrm{NN}},\Delta_U)$ admits a finite incremental $\ell_2$-gain bound $\sqrt{\gamma}$.
\end{itemize}
We also identify additional conditions under which incremental convergence
implies convergence of all state trajectories to a common limit point.

%% file: IQC.tex
\section{Incremental IQC-Based Stability And Performance Analysis}\label{sec:background}

This section develops the full-order incremental IQC certificate used in our
analysis. We first introduce incremental IQCs to characterize the incremental
input--output behavior of the NN and the uncertain perturbation, and then
review the dissipation-based and incremental small-gain conditions used to
analyze the resulting interconnection.

\subsection{Incremental IQCs}\label{incremental IQCs}

We first characterize the incremental behavior of the NN nonlinearity $\Delta_{\mathrm{NN}}$.
For feedforward NNs with activation functions that are slope-restricted in $[0,1]$, the following structured $\mathsf\Lambda$-incremental QC characterization holds.

\paragraph{$\mathsf\Lambda$-incremental QC}
For each $l\in[N]$, let
$h^l = \phi(W_l h^{l-1} + b_l)$
denote the output of layer $l$ as defined in~\eqref{eq:Delta}, where $h^0 = v$. 
We then define the stacked vector $\bold h := \begin{bmatrix}(h^1)^\tp & \cdots &(h^N)^\tp\end{bmatrix}^\tp$,
corresponding to the input $v$. Similarly, $\bold h'$ corresponds to the input $v'$.
Then, for any $v,v'\in\mathbb{R}^{n_v}$ and any collection of diagonal positive semidefinite matrices
$\mathsf{\Lambda}:=\{\Lambda_l\}_{l=1}^N$ with $\Lambda_l\in\mathbb{R}^{n_l\times n_l}$,
the following inequality holds:
\begin{equation}\label{eq:Lambda-IQC}
\bmat{\delta v\\ \delta \bold h}^\tp
\mM_{\mathrm{NN}}(\mathsf\Lambda)
\bmat{\delta v\\ \delta \bold h} \ge 0,
\end{equation}
where $\delta v = v - v^\prime$, $\delta \bold h = \bold h - \bold h^\prime$, and
\begin{align}
\nonumber
 \mM_{\mathrm{NN}}(\mathsf{\Lambda})=\bmat{
    0             & W_1^{\tp}\Lambda_1 & 0              & \cdots & 0                   \\
\Lambda_1 W_1 & -2\Lambda_1         & W_2^{\tp}\Lambda_2 & \ddots & \vdots              \\
0             & \Lambda_2 W_2       & -2\Lambda_2    & \ddots & 0                  \\
\vdots        & \ddots              & \ddots         & \ddots & W_N^{\tp}\Lambda_N \\
0             & \cdots              & 0              & \Lambda_N W_N & -2\Lambda_N }.   
\end{align}
The $\mathsf\Lambda$-incremental QC characterizes the
incremental input--output behavior of the feedforward NN
through a structured QC that explicitly
captures the layer-wise interconnection structure of the
network.
To characterize the uncertain perturbation $\Delta_\text{U}$ in a unified incremental framework,
we adopt the following incremental $\rho$-hard IQC description.

\paragraph{Incremental $\rho$-Hard IQC and System Representation}\label{incremental hard IQCs}

We adopt an incremental $\rho$-hard IQC to characterize the uncertain perturbation $\Delta_\text{U}$.

\begin{definition}[Incremental $\rho$-hard IQC]
Let $\rho \in (0,1]$, 
$\Psi \in \mathbb{RH}_\infty^{n_r \times (n_p+n_q)}$, 
and $M = M^\tp \in \mathbb{R}^{n_r \times n_r}$ be given.
An operator $\Delta_{\operatorname{U}} : \ell_{2e}^{n_p} \to \ell_{2e}^{n_q}$
satisfying Assumption~\ref{assumption1}
is said to satisfy the incremental $\rho$-hard IQC
defined by $(\Psi,M)$ if, for any
$p,p^\prime \in \ell_{2e}^{n_p}$ with
$q = \Delta_{\operatorname{U}}(p)$ and
$q^\prime = \Delta_{\operatorname{U}}(p^\prime)$,
the incremental signals
$\delta p := p-p^\prime$,
$\delta q := q-q^\prime$
satisfy
\begin{equation}\label{eq:inc-hard-IQC}
\sum_{k=0}^{K}
\rho^{-k}\,
\delta r_k^\tp M \delta r_k
\ge 0,
\qquad
\forall K \in \mathbb{Z}_{\ge 0},
\end{equation}
where $\delta r
=
\Psi
\begin{bmatrix}
\delta p\\
\delta q
\end{bmatrix}$
with zero initial condition.
\end{definition}

\begin{remark}[Monotonicity in $\rho$ \cite{Lessard2014}]\label{remarkrho}
If $\Delta_{\operatorname{U}}$ satisfies the incremental $\rho_1$-hard IQC defined by $(\Psi,M)$ for some $\rho_1 \in (0,1]$,
then it also satisfies the incremental $\rho$-hard IQC for any $\rho \in [\rho_1,1]$.
\end{remark}

In this work, we characterize the uncertain perturbation $\Delta_\text{U}$ using time-domain incremental $\rho$-hard IQCs. Specifically, $\Delta_\text{U}$ is assumed to satisfy a collection of IQCs defined by $\{(\Psi_i,M_i)\}_{i=1}^{N_\Psi}$.
The corresponding filters are aggregated into a single filter $\Psi$ with the following state-space realization:
\begin{align}\label{filter}
\begin{split}
    \psi_{k+1} &= A_\Psi \psi_k + B_{\Psi 1} p_k + B_{\Psi 2} q_k\\
    r_k &= C_\Psi \psi_k + D_{\Psi 1} p_k + D_{\Psi 2} q_k\\
    \psi_0 &= 0,
\end{split}
\end{align}
where $\psi_k \in \mathbb{R}^{n_\psi}$ denotes the filter state and
$r_k \in \mathbb{R}^{n_r}$ denotes the aggregated output.
In particular, $r_k := \bmat{r^{(1)\tp}_k & \ldots & r_k^{(N_\Psi)\tp}}^\tp$, where
$r_k^{(i)}$ is the output of the individual filter $\Psi_i$. Define the corresponding incremental filtered outputs by
$\delta r_k^{(i)}
:=
r_k^{(i)}-{r_k^{(i)}}'$,
$i\in[N_\Psi]$ and $\delta r_k := \bmat{\delta r^{(1)\tp}_k & \ldots & \delta r_k^{(N_\Psi)\tp}}^\tp$. We further define the block-diagonal multiplier
$M(\alpha) := \operatorname{diag}(\alpha_1 M_1,\ldots,\alpha_{N_\Psi} M_{N_\Psi})$,
where $\alpha := \{\alpha_i\}_{i=1}^{N_\Psi}$ are nonnegative scalars.
Then the incremental filtered output $\delta r_k$ satisfies
\begin{equation}\label{eq:multi-hard-iqc}
\sum_{k=0}^{K} \rho^{-k} \, \delta r_k^\tp M(\alpha) \delta r_k
= \sum_{i=1}^{N_\Psi} \alpha_i \sum_{k=0}^{K} \rho^{-k}\, \delta r_k^{(i)\tp} M_i \delta r_k^{(i)} \ge 0.
\end{equation}

Let $\zeta_k := \bmat{x^\tp_k & \psi^\tp_k}^\tp \in \mathbb{R}^{n_x+n_\psi}$ denote the extended state vector.
Then, the dynamics of the uncertain interconnected system can be rewritten as \begin{align}\label{unsys}
    \begin{split}
        \bmat{\zeta_{k+1}\\r_k\\e_k} & = \bmat{A&B_1&B_2&B_3\\
        C_1&D_{11}&D_{12}&D_{13}\\
        C_2 & D_{21} & D_{22} & D_{23}} \bmat{\zeta_k\\w_k\\q_k\\d_k},
    \end{split}
\end{align}
where the state space matrices are
\begin{align}
\begin{split}
\nonumber
    &A  = \bmat{A_G & 0\\
    B_{\Psi1}C_{G2} & A_\Psi},\quad B_1 = \bmat{B_{G1}\\B_{\Psi1}D_{G21}}, \\ &B_2  = \bmat{B_{G2}\\ B_{\Psi1}D_{G22} + B_{\Psi2}},\quad B_3  = \bmat{B_{G3}\\B_{\Psi1}D_{G23}},\\
    &C_1  = \bmat{D_{\Psi1}C_{G2}&C_{\Psi}},\quad D_{11} = D_{\Psi1}D_{G21},\\ 
    &D_{12}  =  D_{\Psi1}D_{G22} +  D_{\Psi2},\quad D_{13} = D_{\Psi1}D_{G23},\\ &C_2 = \bmat{C_{G3} & 0}, \quad
    D_{21} = D_{G31},\\ & D_{22} = D_{G32},\quad D_{23} = D_{G33}.
\end{split}
\end{align}

\subsection{Dissipation Analysis}

We integrate the incremental characterizations developed in the previous subsections into a dissipation-based analysis.
This leads to a large-scale SDP sufficient condition for
incremental convergence and for certifying an upper bound on
the incremental $\ell_2$ gain.

For convenience in the subsequent LMI formulation, define matrices
$\mathbf A(P,\gamma,\alpha)$, $\mathbf
 B(P,\alpha)$, $\mathbf C(P,\alpha)$
as follows:
\begin{align}
\begin{split}
\nonumber
    \mathbf{A}(P,\gamma,\alpha)
    :=
    &\bmat{
    A^\tp P A - P & A^\tp P B_2 & A^\tp P B_3 \\
    B_2^\tp P A & B_2^\tp P B_2 & B_2^\tp P B_3 \\
    B_3^\tp P A & B_3^\tp P B_2 & B_3^\tp P B_3 - \gamma I_{n_d}
    }
    \notag\\
    &+
    \bmat{
    C_1^\tp M(\alpha) C_1 &
    C_1^\tp M(\alpha) D_{12} &
    C_1^\tp M(\alpha) D_{13} \\
    D_{12}^\tp M(\alpha) C_1 &
    D_{12}^\tp M(\alpha) D_{12} &
    D_{12}^\tp M(\alpha) D_{13} \\
    D_{13}^\tp M(\alpha) C_1 &
    D_{13}^\tp M(\alpha) D_{12} &
    D_{13}^\tp M(\alpha) D_{13}
    }
    \notag\\
    &+
    \bmat{
    C_2^\tp C_2 &
    C_2^\tp D_{22} &
    C_2^\tp D_{23} \\
    D_{22}^\tp C_2 &
    D_{22}^\tp D_{22} &
    D_{22}^\tp D_{23} \\
    D_{23}^\tp C_2 &
    D_{23}^\tp D_{22} &
    D_{23}^\tp D_{23}
    },
    \notag\\[1mm]
    \mathbf{B}(P,\alpha)
    :=
    &
    \bmat{
    A^\tp P B_1 \\
    B_2^\tp P B_1 \\
    B_3^\tp P B_1
    }
    +
    \bmat{
    C_1^\tp M(\alpha) D_{11} \\
    D_{12}^\tp M(\alpha) D_{11} \\
    D_{13}^\tp M(\alpha) D_{11}
    }
    +
    \bmat{
    C_2^\tp D_{21} \\
    D_{22}^\tp D_{21} \\
    D_{23}^\tp D_{21}
    },
    \notag\\[1mm]
    \mathbf{C}(P,\alpha)
    :=
    &
    B_1^\tp P B_1
    +
    D_{11}^\tp M(\alpha) D_{11}
    +
    D_{21}^\tp D_{21}
    \end{split}
\end{align}

In addition, to simplify the subsequent analysis, we introduce an equivalent NN representation
$\Delta'_{\mathrm{NN}}$
with augmented input $\xi_k := \begin{bmatrix}
\zeta_k^\tp & q_k^\tp & d_k^\tp
\end{bmatrix}^\tp \in \mathbb{R}^{m},\
m := n_x+n_\psi+n_q+n_d$:
\begin{align}\label{eq:Delta2}
\begin{split}
h_k^0 &= \xi_k,\\
h_k^1 &= \phi\!\left(\widetilde{W}_1 h_k^{0} + b_1\right),\\
h_k^l &= \phi\!\left(W_l h_k^{l-1} + b_l\right), \quad l = 2,\dots,N,\\
w_k &= h_k^N,
\end{split}
\end{align}
where $\widetilde{W}_1
    :=
    W_1
    \bmat{
        C_{G1} &
        0_{n_v \times n_\psi} & 
        D_{G12} &
        D_{G13}
    }$, which incorporates the affine relation $v_k =  \bmat{
        C_{G1} &
        0_{n_v \times n_\psi} & 
        D_{G12} &
        D_{G13}
    }\xi_k$
into the first layer. Compared with $\Delta_{\text{NN}}$ \eqref{eq:Delta}, $\Delta_{\text{NN}}'$ \eqref{eq:Delta2} differs only in its first layer through the use of $\widetilde{W}_1$. The remaining layers are identical.

We next state the resulting large-scale SDP condition.

\begin{theorem}\label{Theorem:largeSDP1}
Let $G$ be the LTI plant defined by \eqref{eq:uncertain_sys}, 
let $\Delta_{\mathrm{NN}}$ denote the NN defined in \eqref{eq:Delta}, 
and let
$\Delta_{\operatorname{U}} : \ell_{2e}^{n_p} \to \ell_{2e}^{n_q}$
be an uncertain perturbation satisfying Assumption~\ref{assumption1}
and the incremental $\rho$-hard IQCs defined by
$\{(\Psi_i,M_i)\}_{i=1}^{N_\Psi}$.
Suppose that there exist a positive definite matrix $P$, a scalar $\gamma > 0$, nonnegative scalars $\alpha := \{\alpha_i\}_{i=1}^{N_\Psi}$, and diagonal positive semidefinite matrices $\mathsf{\Lambda}:=\{\Lambda_l\}_{l=1}^N$ such that the following LMI condition holds:
\begin{align}\label{eq:robust-condition}
\widetilde{\mathsf{M}}_{\operatorname{NN}}(\mathsf\Lambda)
    +
    \mathsf{M}_G(P,\gamma,\alpha)
    \prec 0,
 \end{align}
 where 
 \begin{align}\label{tildeNN}
&\widetilde{\mathsf{M}}_{\operatorname{NN}}(\mathsf\Lambda)
=
\scalebox{0.9}{$
\bmat{
0_{m}
&
\widetilde{W}_1^\tp \Lambda_1
&
0
&
\cdots
&
0
\\
\Lambda_1 \widetilde{W}_1
&
-2\Lambda_1
&
W_2^\tp \Lambda_2
&
\ddots
&
\vdots
\\
0
&
\Lambda_2 W_2
&
-2\Lambda_2
&
\ddots
&
0
\\
\vdots
&
\ddots
&
\ddots
&
\ddots
&
W_N^\tp \Lambda_N
\\
0
&
\cdots
&
0
&
\Lambda_N W_N
&
-2\Lambda_N
},
$} \\
  &\mathsf{M}_G(P, \gamma,\alpha)=\bmat{\mathbf{A}(P,\gamma,\alpha) & 0 & \cdots & \mathbf{B}(P,\alpha) \\ 0 & 0 & \cdots & 0\\ \vdots & \vdots & \ddots & \vdots \\ \mathbf{B}(P,\alpha)^\tp  & 0 & \cdots & \mathbf{C}(P,\alpha)}.\label{ABC}
\end{align}
Then 
\begin{itemize}
    \item  If $\rho \in (0,1]$ and the external input $d = 0$, then the closed-loop trajectories are incrementally convergent, i.e.,
for any two trajectories initialized from $x_0,x_0' \in \mathbb{R}^{n_x}$, $\|x_k-x_k'\|_2 \to 0 \ \text{as } k\to\infty$.
    \item If $\rho \in (0,1]$, then the incremental $\ell_2$ gain of the system is bounded by $\sqrt{\gamma}$.
\end{itemize}
\end{theorem}

A brief derivation and interpretation of the LMI structure in
\eqref{eq:robust-condition} are provided in Appendix~\ref{whyLMI}. Note that the dimension of 
$\widetilde{\mM}_{\text{NN}}(\mathsf{\Lambda})$ grows with the depth of the NN,
which renders the resulting SDP computationally and memory intensive,
and therefore not scalable to deep NNs.
In addition, applying a chordal sparsity decomposition to
\eqref{eq:robust-condition} still yields an SDP condition whose
size scales with the depth of the NN component
\cite{xue2022chordal,xue2024chordal}.

\subsection{Incremental Small Gain Method}

In addition to the dissipation-based condition in Theorem 1,
the incremental small-gain method provides an alternative condition for establishing incremental convergence and incremental $\ell_2$-gain bounds.

\begin{proposition}[Incremental small gain method]\label{proposition2} Under the assumptions of Theorem~\ref{Theorem:largeSDP1}, suppose that $\sqrt{L}$ is an upper bound on the
Lipschitz constant of $\Delta_{\mathrm{NN}}'$ \eqref{eq:Delta2}, i.e.,
$\|w-w^\prime\|_2
\le
\sqrt{L}\|\xi-\xi^\prime\|_2$,
where
$w=\Delta'_{\mathrm{NN}}(\xi)$
and $w^\prime=\Delta'_{\mathrm{NN}}(\xi^\prime)$,
for all
$\xi,\xi^\prime\in\mathbb R^{m}$.
If there exist a positive definite matrix $P$, a scalar $\gamma > 0$,
a collection of nonnegative scalars $\alpha$
such that the following LMI condition holds:
\begin{align}\label{eq:smallgain}
\begin{split}
&\bmat{\mathbf{A}(P,\gamma,\alpha) &\mathbf{B}(P,\alpha)\\
\mathbf{B}(P,\alpha)^\tp &\mathbf{C}(P,\alpha)} + \operatorname{diag}(I_{m},-\tfrac{1}{L}I_{n_w}) \prec 0.
\end{split}
\end{align}
Then the conclusions of Theorem~\ref{Theorem:largeSDP1} follow.
\end{proposition}

\begin{remark}
Proposition~\ref{proposition2} applies the incremental small-gain method to the feedback interconnection between $\Delta'_{\mathrm{NN}}$ and an augmented system consisting of $G$, $\Delta_{\operatorname{U}}$, the exogenous input channel $d$, and the performance output channel $e$.

In the special case where $\Delta_{\operatorname{U}}$ and the external
channels are absent, the interconnection reduces to the
standard feedback loop between $G$ and
$\Delta_{\mathrm{NN}}$, where the NN input coincides with
the state (i.e., $v=x$), as considered in
\cite{Wang2025ScalableDF}.
In this case, Proposition~\ref{proposition2} specializes to the classical
incremental small-gain method (see Proposition~1 in
\cite{Wang2025ScalableDF}), and the LMI condition in
Proposition~\ref{proposition2} simplifies to
\begin{align}\label{eq:KYP}
\bmat{A_G^\tp P A_G - P & A_G^\tp P B_{G1}\\
      B_{G1}^\tp P A_G & B_{G1}^\tp P B_{G1}}
+ \operatorname{diag}(I_{n_v},-\tfrac{1}{L}I_{n_w})
\prec 0.
\end{align}
\end{remark}

The dimension of the LMI in Proposition~\ref{proposition2}
depends only on the state-space matrices in \eqref{unsys} and the
width of the last layer of the NN (i.e., $n_w$), and is independent of
the network depth.
However, it may be overly conservative, as it does not fully exploit the structural coupling between $G$ and $\Delta_{\mathrm{NN}}'$.
The required Lipschitz bound can be obtained using Lipschitz constant estimation algorithms,
which are introduced in the following section.

%% file: Scalable.tex
\section{Scalable Lipschitz Constant Estimation for Deep NNs}\label{scalablesection}

This section reviews scalable methods for certifying Lipschitz
bounds of deep NNs. We first introduce the LipSDP framework and
then discuss its scalable variants, which will be used in the
subsequent development of our scalable stability conditions.

\noindent \textbf{LipSDP.} 
For a $\Delta_{\text{NN}}$ as defined in \eqref{eq:Delta} with activation functions that are slope-restricted in $[0,1]$, the LipSDP framework
\cite{fazlyab2019efficient,pauli2021training,pauli2022neural}
certifies tight upper bounds on the Lipschitz constant and has inspired a large body of subsequent work on Lipschitz NNs
\cite{10179161,araujo2023a,wang2022,havens2023exploiting,pauli2023novel,barbara2024robust,wang2023direct}.
LipSDP certifies that the $\Delta_{\text{NN}}$ is $\sqrt{L}$-Lipschitz if there exists a scalar $L > 0$ and a collection of diagonal positive semidefinite matrices
$\mathsf{\Lambda}$, such that the following LMI condition holds:
\begin{align}\label{eq:LipSDP}
\mM_{\mathrm{NN}}(\mathsf{\Lambda})
+ \text{diag}(-I_{n_v},0_{n_\phi - n_w},\tfrac{1}{L}I_{n_w})
\prec 0.
\end{align}
Equivalently, LipSDP can be viewed as an optimization problem that minimizes $L$ subject to the matrix inequality~\eqref{eq:LipSDP}.
For a given NN with weights $\{W_l\}_{l=1}^N$, the constraint~\eqref{eq:LipSDP}
is affine in the decision variables $\mathsf{\Lambda}$ and $\lambda := 1/L$,
resulting in a SDP that maximizes $\lambda$ (equivalently, minimizes $L$). Since $\Delta'_{\mathrm{NN}}$ differs from $\Delta_{\mathrm{NN}}$ only in the first layer, the same LipSDP formulation applies after replacing $\mM_{\mathrm{NN}}(\mathsf{\Lambda})
+ \text{diag}(-I_{n_v},0_{n_\phi - n_w},\tfrac{1}{L}I_{n_w})$ with $\widetilde{\mM}_{\mathrm{NN}}(\mathsf{\Lambda})
+ \text{diag}(-I_{m},0_{n_\phi - n_w},\tfrac{1}{L}I_{n_w})$.
While LipSDP often provides tight Lipschitz upper bounds, the dimension of the associated LMI grows with the network depth. As a result, the resulting SDP can become computationally and memory intensive for deep NNs.

Motivated by the limited scalability of LipSDP, a variety of scalable methods
\cite{xueclipse,wang2024scalability,syed2025improved}
have been proposed. These methods construct structured feasible points for the LipSDP LMI \eqref{eq:LipSDP} to obtain efficient approximations of the Lipschitz bound. 

\vspace{0.5em}

\noindent\textbf{ECLipsE~\cite{xueclipse}.}
ECLipsE constructs a feasible point for the LipSDP LMI \eqref{eq:LipSDP} by solving $N$ small-scale SDPs.
The method initializes with $\Xi_1=-I_{n_v}$.
For each $l\in[N]$, the matrix $\Lambda_l$ is computed by solving the following SDP:
\begin{equation}\label{eq:ECLipsE-unified-SDP}
\begin{aligned}
\max_{c_l,\Lambda_l}\quad & c_l \\
\text{s.t.}\quad &
\begin{bmatrix}
-2\Lambda_l + c_l W_{l+1}^\tp W_{l+1} & \Lambda_l W_l \\
W_l^\tp \Lambda_l & \Xi_l
\end{bmatrix} \prec 0,\\
& \Lambda_l \succeq 0,\quad c_l > 0,
\end{aligned}
\end{equation}
where we define $W_{N+1}=I_{n_w}$.
For $l = 2,\cdots ,N$, the matrix $\Xi_l$ is updated as
\begin{equation}\label{eq:ECLipsE-recursion}
\Xi_l = -2\Lambda_{l-1}
- \Lambda_{l-1} W_{l-1} \Xi_{l-1}^{-1} W_{l-1}^\tp \Lambda_{l-1}.
\end{equation}
Then $\Delta_{\mathrm{NN}}$ is certified to be $\sqrt L$-Lipschitz for any $L>0$ satisfying
\begin{equation}\label{eq:Lipschitz-cond}
\tfrac{1}{L} I_{n_w} - 2\Lambda_N - \Lambda_N W_N \Xi_N^{-1} W_N^\tp \Lambda_N \prec 0.
\end{equation}
The memory footprint of ECLipsE does not scale with the network depth, since it stores and updates only a constant number of matrices per layer. Its computational complexity, however, grows linearly with the network depth because one small SDP must be solved at each layer. Empirical results indicate that ECLipsE achieves Lipschitz upper bounds that are nearly as tight as those produced by LipSDP \cite{xueclipse}.

\vspace{0.5em}

\noindent\textbf{ECLipsE-Fast~\cite{xueclipse}.} To further improve scalability, ECLipsE-Fast constructs a feasible point of LipSDP LMI \eqref{eq:LipSDP} that shares the recursive structure of ECLipsE while avoiding per-layer SDPs \cite{xueclipse}. The method initializes with $\Xi_1 = -I_{n_v}$.
 For each $l \in [N]$, ECLipsE-Fast computes $\Lambda_l$ in closed form based on $\Xi_{l}$ as
\begin{align}\label{update:ECLipsE-Fast}
\Lambda_l = \frac{1}{\norm{W_l \Xi_{l}^{-1} W_l^\tp}_2}\, I_{n_l}.
\end{align}
The matrix $\Xi_l$ is then updated according to~\eqref{eq:ECLipsE-recursion}.
The above construction yields a $\sqrt{L}$-Lipschitz upper bound for the $\Delta_{\text{NN}}$, where $L$ satisfies~\eqref{eq:Lipschitz-cond}.
In practice, both the memory footprint and runtime of ECLipsE-Fast are nearly insensitive to the network depth. This is because the per-layer computation incurs negligible overhead and only requires inexpensive spectral norm approximations. Such approximations can be efficiently computed using Gram iteration \cite{delattre2023efficient}.
However, this computational efficiency typically comes at the cost of increased conservatism, and ECLipsE-Fast typically produces looser Lipschitz bounds than ECLipsE in practice \cite{xueclipse}.

\vspace{0.5em}

\noindent\textbf{Variants of ECLipsE-Fast~\cite{syed2025improved}.}
Motivated by the efficiency of ECLipsE-Fast, several improved variants have been proposed to enhance the practical tightness of the resulting Lipschitz bounds while preserving scalability. 
These methods retain the recursive structure of ECLipsE-Fast while employing different closed-form update rules for $\Lambda_l$. Representative examples include ECLipsE-SN, which uses spectral-normalization-based scaling, and ECLipsE-GC/GCS, which leverage Gershgorin-type bounds with diagonal scaling.
Another variant, ECLipsE-Shift, introduces additional diagonal shifts to reduce conservatism. All these variants construct feasible points for the LipSDP LMI without solving per-layer SDPs, thereby inheriting the low memory footprint and computational efficiency of ECLipsE-Fast while maintaining scalability with respect to the network depth. Empirically, these variants often produce tighter Lipschitz bounds than ECLipsE-Fast~\cite{syed2025improved}.

\vspace{0.5em}

\noindent\textbf{EP-LipSDP~\cite{wang2024scalability}.}
While ECLipsE-type methods provide highly efficient feasible-point constructions, they do not guarantee optimality with respect to LipSDP. EP-LipSDP addresses this limitation by reformulating LipSDP as an eigenvalue-based optimization problem while preserving the optimal Lipschitz bound. This reformulation enables the use of first-order optimization methods, substantially reducing the memory requirements compared with interior-point SDP solvers. Moreover, optimal solutions of EP-LipSDP can be directly mapped to optimal solutions of LipSDP~\cite{wang2024scalability}.

%% file: Normal_stability_verification_modified.tex
\section{Scalable Stability And Performance Analysis}\label{sec:normal}

In this section, we develop a scalable framework for stability and performance analysis by leveraging the Lipschitz estimation methods introduced in Section~\ref{scalablesection}. The key idea is to decompose the original LMI condition in~\eqref{eq:robust-condition} so that part of the decision variables can be determined using scalable Lipschitz estimation algorithms, thereby reducing the dimension of the remaining optimization problem. Building on this decomposition, we further introduce a multi-step update scheme to reduce conservatism.

\subsection{Matrix Decomposition for Reducing LMI Dimension}

We start with the following useful lemma, which extracts the weight matrices $\{W_l\}_{l=1}^N$ and $\mathsf\Lambda$ from the original large-scale LMI~\eqref{eq:robust-condition}.

\begin{lemma}\label{lemma2}
  The LMI \eqref{eq:robust-condition} holds if and only if there exist a positive definite matrix $P$, diagonal positive semidefinite matrices $\mathsf{\Lambda}$, symmetric matrices $Y_1,Y_2$, matrix $Y_3$, a collection of nonnegative scalars $\alpha$, and a scalar $\gamma>0$ such that the following coupled LMIs hold:
  \begin{align}
& \bmat{\mathbf{A}(P,\gamma,\alpha) - Y_1 &\mathbf{B}(P,\alpha) - Y_3\\
\mathbf{B}(P,\alpha)^\tp - Y_3^\tp &\mathbf{C}(P,\alpha) - Y_2 } \prec 0\label{eq:key2}\\[1.5ex]
&\widetilde{\mathsf{M}}_{\operatorname{NN}}(\mathsf{\Lambda}) + \mathsf{M}(Y_1,Y_2,Y_3)\prec 0,\label{eq:key3}
  \end{align}
  where  
  \[
  \mathsf{M}(Y_1,Y_2,Y_3) = \bmat{Y_1 & 0 & \cdots & Y_3 \\ 0 & 0 & \cdots & 0\\ \vdots & \vdots & \ddots & \vdots \\ Y_3^\tp  & 0 & \cdots & Y_2}.
  \]
\end{lemma}

Lemma \ref{lemma2} decomposes the original large-scale LMI \eqref{eq:robust-condition} into the coupled LMIs \eqref{eq:key2} and \eqref{eq:key3}. 
The LMI \eqref{eq:key2} depends only on the state-space matrices in \eqref{unsys} and its size is independent of the NN dimensions, whereas the size of LMI \eqref{eq:key3} scales with the width and depth of the NN.
 The two LMIs are coupled through the auxiliary variables $\{Y_1, Y_2, Y_3\}$.
In particular, if there exist $\{P,\alpha,\gamma,\mathsf\Lambda,Y_1,Y_2,Y_3\}$ satisfying \eqref{eq:key2} and \eqref{eq:key3}, then the original LMI condition \eqref{eq:robust-condition} holds with the same choice of $\{P,\alpha,\gamma,\mathsf\Lambda\}$. Consequently, all conclusions of Theorem~\ref{Theorem:largeSDP1} follow.

\begin{remark}[Small-gain method as a special case]
When the Lipschitz bound $\sqrt{L}$ for $\Delta'_{\mathrm{NN}}$ is certified by the corresponding LipSDP condition, the small-gain condition in Proposition 1 can be recovered from Lemma 1 using a particular parameterization of the auxiliary variables $Y_1,Y_2,Y_3$. Specifically, consider the choice $Y_1 = -I_{m}$, 
$Y_2 = \tfrac{1}{L}I_{n_w}$, 
$Y_3 = 0$.
Under this parameterization, \eqref{eq:key2} coincides with the incremental small-gain condition in~\eqref{eq:smallgain}, while \eqref{eq:key3} reduces to the LipSDP condition associated with $\Delta_{\mathrm{NN}}'$. Consequently, the small-gain-based stability condition in Proposition~\ref{proposition2} can be viewed as a special case of the decomposition framework introduced in Lemma~\ref{lemma2}.
\end{remark}

Under the parameterization
$Y_1 = -I_{m}$, $Y_2 = \tfrac{1}{L}I_{n_w}$, and $Y_3 = 0$,
the resulting small-gain condition admits a highly scalable implementation. 
In particular, any scalable Lipschitz estimation algorithm introduced in Section~\ref{scalablesection} directly provides a feasible point for the LipSDP LMI. While this choice leads to an efficient stability verification procedure, it restricts $Y_1$ and $Y_2$ to scaled identity matrices and fixes $Y_3 = 0$, which may limit the achievable tightness of the resulting condition. Allowing $\{Y_1,Y_2,Y_3\}$ to be general decision variables can therefore reduce conservatism. Motivated by this observation, we next develop a unified framework that systematically reduces the dimensionality of the NN-dependent LMI condition \eqref{eq:key3}.

\begin{theorem}\label{coro:main3}
Suppose that $\{\mathsf{\Lambda}_{N-1},Y_1\}$ satisfy
$\widetilde{\mathsf{M}}_{\operatorname{NN}}(\mathsf{\Lambda}_{N-1}) + \operatorname{diag}(Y_1,0_{n_\phi-n_w}) \prec 0$, where $\mathsf{\Lambda}_{N-1} := \{\Lambda_{l}\}_{l=1}^{N-1}$ and
\[
\widetilde{\mathsf{M}}_{\operatorname{NN}}(\mathsf\Lambda_{N-1})
=
\scalebox{0.75}{$
\bmat{
0_{m}
&
\widetilde{W}_1^\tp \Lambda_1
&
0
&
\cdots
&
0
\\
\Lambda_1 \widetilde{W}_1
&
-2\Lambda_1
&
W_2^\tp \Lambda_2
&
\ddots
&
\vdots
\\
0
&
\Lambda_2 W_2
&
-2\Lambda_2
&
\ddots
&
0
\\
\vdots
&
\ddots
&
\ddots
&
\ddots
&
W_{N-1}^\tp \Lambda_{N-1}
\\
0
&
\cdots
&
0
&
\Lambda_{N-1} W_{N-1}
&
-2\Lambda_{N-1}
}.
$}
\]
Define $\Xi_1 = Y_1$ and $\Gamma_1 = I_{m}$, and define
\begin{align}\label{Xi2Gamma2}
    \Xi_2 &= - 2\Lambda_{1}
    -\Lambda_{1}
    \widetilde{W}_{1}\Xi_{1}^{-1}\widetilde{W}_{1}^\tp
    \Lambda_{1}, \nonumber\\
    \Gamma_2 &= -\Lambda_{1}
    \widetilde{W}_{1}\Xi_{1}^{-1}\Gamma_{1}.
\end{align}
For $l = 3,\ldots,N$, recursively define $\Xi_l$ according to \eqref{eq:ECLipsE-recursion} and $\Gamma_l$ according to 
\[
\Gamma_l = -\Lambda_{l-1}
    W_{l-1}\Xi_{l-1}^{-1}\Gamma_{l-1}.
\]
If there exist a positive definite matrix $P$, 
a diagonal positive semidefinite matrix $\Lambda_N$, 
symmetric matrix $Y_2$, matrix $Y_3$, 
a scalar $\gamma > 0$, 
and a collection of nonnegative scalars 
$\boldsymbol{\alpha}$ 
such that the following LMIs hold:
\begin{align}\label{key:lmi5}
\begin{split}
& \bmat{\mathbf{A}(P,\gamma,\alpha) - Y_1 &\mathbf{B}(P,\alpha) - Y_3\\
\mathbf{B}(P,\alpha)^\tp - Y_3^\tp &\mathbf{C}(P,\alpha) - Y_2 } \prec 0\\[1.5ex]
  &\scalebox{0.95}{$
    \bmat{
        \Xi_N 
        & W_N^\tp\Lambda_N + \Gamma_N Y_3 
        & 0\\
        \Lambda_NW_N + Y_3^\tp \Gamma_N^\tp 
        & Y_2 - 2\Lambda_N 
        & Y_3^\tp\\
        0 
        & Y_3 
        & \Big(\sum_{l=1}^{N-1} \Gamma_l^\tp \Xi_l^{-1} \Gamma_l \Big)^{-1}
    }$}\prec 0,
\end{split}
\end{align}
then LMIs \eqref{eq:key2} and \eqref{eq:key3} are satisfied with the same
$\{P,\boldsymbol{\alpha},\gamma,\mathsf{\Lambda},Y_1,Y_2,Y_3\}$, and consequently all conclusions of Theorem~\ref{Theorem:largeSDP1} remain valid.
\end{theorem}

\begin{table*}[t]
\centering
\caption{Comparison of the dimensions and structural properties of
representative verification conditions.}
\label{tab:complexity_comparison}
\renewcommand{\arraystretch}{1.15}
\begin{tabular}{llll}
\hline
Method
& Main control-analysis LMI dimension
& Depth dependence
& Coupling retained \\
\hline

Full-order SDP (Theorem \ref{Theorem:largeSDP1})
& $m+\sum_l n_l$
& Yes
& Full layerwise structure \\

Chordal-DeepSDP \cite{xue2024chordal}
& Small cliques; number grows with $N$
& Yes
& Full structure via sparse decomposition \\

Small gain (Proposition \ref{proposition2})
& $m+n_w$
& Via Lipschitz estimation only
& Global scalar Lipschitz bound \\

Reduced condition (Theorem \ref{coro:main3})
& $m+n_{N-1}+n_w$
& Via Lipschitz estimation only
& General plant--NN coupling via $Y_1,Y_2,Y_3$ \\

\hline
\end{tabular}
\end{table*}

Theorem~\ref{coro:main3} provides a scalable sufficient reduction of the
NN-dependent LMI condition~\eqref{eq:key3} after fixing an upstream
NN certificate. The verification procedure consists of two main steps. First, we seek a feasible triplet $\{\mathsf\Lambda_{N-1}, Y_1, L\}$ satisfying 
$\widetilde{\mathsf M}_{\mathrm{NN}}(\mathsf\Lambda_{N-1}) + \operatorname{diag}\!\left(Y_1, 0_{n_\phi-n_{N-1}-n_{w}}, \tfrac{1}{L} I_{n_{N-1}}\right) \prec 0$, 
with a small feasible value of $L$. The scalar $\sqrt{L}$ can be interpreted as a Lipschitz upper bound for
the first $N-1$ layers of $\Delta'_{\rm NN}$. For a general $Y_1 \prec 0$,
$\sqrt{L}$ is the Euclidean Lipschitz bound of the transformed network
under the input coordinate $\hat{\xi}=R\xi$, equivalently a weighted-input
Lipschitz bound for the original network; see Remark~\ref{remark4}.

Since $\tfrac1L I_{n_{N-1}} + \Xi_N \prec 0$
implies $\Xi_N \prec -\tfrac1L I_{n_{N-1}}$,
a smaller value of $L$
imposes a more negative upper bound on $\Xi_N$. To gain further intuition, consider the special case $Y_3=0$, under which the second LMI in \eqref{key:lmi5} reduces to 
\[
\begin{bmatrix} \Xi_N & W_N^\tp\Lambda_N   \\  \Lambda_N W_N & Y_2 - 2\Lambda_N \end{bmatrix} \prec 0.
\]
By the Schur complement, the above LMI is equivalent to
\[
\Xi_N - W_N^\tp\Lambda_N \left(Y_2 - 2\Lambda_N\right)^{-1} \Lambda_N W_N \prec 0. 
\]
Since $\left(Y_2 - 2\Lambda_N\right)^{-1} \prec 0$, a more negative $\Xi_N$ makes the Schur complement inequality easier to satisfy. 
This relation indicates that a smaller $L$, which enforces a more
negative upper bound on $\Xi_N$, can potentially reduce the
conservatism of condition~\eqref{key:lmi5}.

Owing to the computational efficiency and low memory footprint of the Lipschitz estimation algorithms introduced in Section~\ref{scalablesection}, the first step can be carried out efficiently even for large-scale NNs. Besides, different algorithms in Section~\ref{scalablesection} induce different tradeoffs between conservatism, computational complexity, and memory requirements.
 Algorithms that produce tighter Lipschitz upper bounds, such as EP-LipSDP, typically yield less conservative stability conditions at the expense of increased computational and memory costs. In contrast, lightweight methods such as ECLipsE-Fast and its variants substantially reduce the computational and memory cost, but usually produce looser Lipschitz bounds and therefore more conservative conditions. This tradeoff will also be illustrated in our numerical experiments.

Second, based on this feasible point, the matrices $\Xi_N$ and $\Gamma_N$ are constructed via the proposed recursion, and the LMIs in \eqref{key:lmi5} are subsequently verified. 
Importantly, checking \eqref{key:lmi5} is computationally efficient. The sizes of these LMIs are independent of the depth of the NN and depend only on the dimensions of the last two layers (i.e., $n_{N-1}$ and $n_w$) and the state-space matrices in \eqref{unsys}.

Table~\ref{tab:complexity_comparison} summarizes the scalability and structural tradeoffs among representative verification approaches in terms of the main control-analysis LMI dimension, depth dependence, and retained plant--NN coupling.

However, the condition in Theorem~\ref{coro:main3} may still be conservative, as $Y_1$ is inherited from the initial Lipschitz certificate and subsequently fixed in the verification of \eqref{key:lmi5}.
In the following subsection, we introduce a multi-step update rule that, when combined with Theorem~\ref{coro:main3}, jointly refines $Y_1$, $Y_2$, and $Y_3$ to further reduce conservatism. 

\subsection{Multi-Step Update Rules for Reducing Conservatism}\label{sec:multistep}

Our approach follows a three-step alternating update procedure. 
In the first step, we initialize $Y_1 \prec 0$ and compute a feasible $\mathsf{\Lambda}_{N-1}$ using a selected scalable Lipschitz constant estimation algorithm. 

\begin{remark}\label{remark4}
To apply the scalable Lipschitz estimation algorithms reviewed in
Section~\ref{scalablesection} for a general $Y_1\prec0$, let $R$ be any nonsingular
matrix satisfying $R^\tp R=-Y_1$.
Introducing the transformed input $\hat{\xi}=R\xi$, the first-layer
weight matrix becomes
$\widehat W_1=\widetilde W_1R^{-1}$.
The standard $-I$-initialized scalable Lipschitz routine can then be
applied to the transformed network. Indeed,
$-\hat{\xi}^\tp\hat{\xi}
    =-\xi^\tp R^\tp R\xi
    =\xi^\tp Y_1\xi$,
so, under the inverse coordinate transformation, the resulting
certificate corresponds to one initialized with the prescribed
$Y_1$. Hence, throughout the following development, a scalable
Lipschitz routine with a general $Y_1\prec0$ refers to applying the
standard routine to the transformed first-layer matrix
$\widehat W_1=\widetilde W_1R^{-1}$. 
\end{remark}

In the second step, with $Y_1$ and $\mathsf{\Lambda}_{N-1}$ fixed, we update $\Lambda_N$, $Y_2$, and $Y_3$. 
In the third step, with $Y_2$, $Y_3$, and $\mathsf{\Lambda}$ fixed, we update $Y_1$.

Suppose that $\gamma > 0$ and $Y_1 \prec 0$ are fixed, and a feasible $\mathsf{\Lambda}_{N-1}$ has been obtained using a scalable Lipschitz constant estimation algorithm.
One convenient way to check the feasibility of the LMIs in \eqref{key:lmi5} is to introduce an auxiliary scalar decision variable
$\nu$ and solve the following optimization problem, which minimizes $\nu$
subject to LMI constraints, with $\{P,\alpha,\nu,\Lambda_N,Y_2,Y_3\}$ as the decision variables:
\begin{align}\label{eq:lmi4}
\begin{split}
& \bmat{\mathbf{A}(P,\gamma,\alpha) - Y_1 &\mathbf{B}(P,\alpha) - Y_3\\
\mathbf{B}(P,\alpha)^\tp - Y_3^\tp &\mathbf{C}(P,\alpha) - Y_2 } \preceq \nu I\\[1.5ex]
  &\scalebox{0.95}{$
    \bmat{
        \Xi_N 
        & W_N^\tp\Lambda_N + \Gamma_N Y_3 
        & 0\\
        \Lambda_NW_N + Y_3^\tp \Gamma_N^\tp 
        & Y_2 - 2\Lambda_N 
        & Y_3^\tp\\
        0 
        & Y_3 
        & \Big(\sum_{l=1}^{N-1} \Gamma_l^\tp \Xi_l^{-1} \Gamma_l \Big)^{-1}
    }$}\prec 0,
\end{split}
\end{align}
If the optimal value $\nu$ satisfies $\nu < 0$, then the LMIs in \eqref{key:lmi5} are feasible. 
Consequently, the resulting choice of $\{P,\alpha,\gamma, \mathsf{\Lambda}, Y_1, Y_2, Y_3\}$ also provides a feasible solution to the original conditions \eqref{eq:key2} and \eqref{eq:key3} in Lemma \ref{lemma2}.

If the optimal value $\nu$ satisfies $\nu \ge 0$, then the corresponding
solution $\{P,\alpha,\gamma,\mathsf{\Lambda},Y_1,Y_2,Y_3\}$ does not yield a feasible solution
to the original LMIs~\eqref{eq:key2} and~\eqref{eq:key3}.
In this case, one can fix $\mathsf{\Lambda}$, $Y_2$, and $Y_3$, and re-optimize
over $Y_1$, $P$, and the auxiliary variable $\nu$.
This re-optimization step amounts to solving a reduced SDP and guarantees
that the resulting value of $\nu$ is less than or equal to the original one.

When $\{\mathsf{\Lambda},Y_2,Y_3\}$ are fixed, the size of the LMI condition
$\widetilde{\mathsf{M}}_{\operatorname{NN}}(\mathsf{\Lambda}) + \mathsf{M}(Y_1,Y_2,Y_3)\prec 0$
can be reduced by recursively applying the Schur complement lemma $N$ times,
leading to an equivalent constraint of the form $Y_1 \prec \Theta$,
where $\Theta$ depends only on $\{\mathsf{\Lambda},Y_2,Y_3\}$ and
$\{W_l\}_{l=1}^N$.
Specifically, we initialize $\Xi_1' = - 2\Lambda_N + Y_2$ and $\Gamma_1' = I_{n_w}$.
The matrices $\Xi_l'$ and $\Gamma_l'$ for $l = 2,\ldots,N$ are recursively computed as
\begin{align}\label{reverse3N}
\begin{split}
\Xi_l' &= - 2\Lambda_{N-l+1}
- W_{N-l+2}^\tp \Lambda_{N-l+2} {(\Xi_{l-1}')}^{-1}
\Lambda_{N-l+2} W_{N-l+2},\\
\Gamma_l' &= - \Gamma_{l-1}'{(\Xi_{l-1}')}^{-1}
\Lambda_{N-l+2}W_{N-l+2}.
\end{split}
\end{align}
Then, $\Theta$ can be rewritten as:
\begin{align}
\begin{split}
\nonumber
\Theta = & \sum_{l=1}^{N-1} \big(Y_3\Gamma_l' {(\Xi_l')}^{-1} (Y_3\Gamma_l')^\tp \big)\\ & + \big(\widetilde{W}_1^\tp\Lambda_1 + Y_3\Gamma_N'\big) {(\Xi_N')}^{-1} \big(\widetilde{W}_1^\tp\Lambda_1 + Y_3\Gamma_N'\big)^\tp. 
\end{split}
\end{align}
Therefore, we can fix $\{\mathsf{\Lambda},Y_2, Y_3\}$ and minimize $\nu$ over the following LMI condition with $\{P, \alpha, \nu, Y_1\}$ being the 
decision variables:
\begin{align}\label{eq:lmi6}
    \begin{split}    
\bmat{\mathbf{A}(P,\gamma,\alpha) - Y_1 &\mathbf{B}(P,\alpha) - Y_3\\
\mathbf{B}(P,\alpha)^\tp - Y_3^\tp &\mathbf{C}(P,\alpha) - Y_2 } \preceq & \nu I\\[0.5em]
Y_1\prec & \Theta.
    \end{split}
\end{align}

\begin{algorithm}[t] 
\caption{Multi-round ABCD Algorithm} 
\label{al:ABCD} 
\begin{algorithmic}[1] 
\REQUIRE state space matrices in \eqref{unsys}, $\gamma > 0$, 
$\{W_l\}_{l=1}^N$, maximum number of rounds $K_{\max}$

\STATE \textbf{Initialize} $Y_1 \prec 0$, $\nu \leftarrow 0$, $k \leftarrow 0$

\WHILE{$\nu \ge 0$}
    \STATE $k \leftarrow k+1$
    
    \STATE \textbf{$\{\mathsf\Lambda_{N-1}\}$-Initialization:} 
    Compute feasible $\mathsf{\Lambda}_{N-1}$ using a scalable 
    Lipschitz constant estimation algorithm.
    
    \STATE \textbf{$\{Y_2,Y_3\}$-Update:} 
    Fix $Y_1$ and $\mathsf{\Lambda}_{N-1}$. Minimize $\nu$ 
    subject to the LMI \eqref{eq:lmi4}, with 
    $\{P,\alpha,\nu,\Lambda_N,Y_2,Y_3\}$ being the decision variables.
    
    \IF{$\nu < 0$}
        \STATE \textbf{Return:} \textsc{Certified} and the corresponding 
        matrices $\{P,\alpha,\gamma,\mathsf{\Lambda},Y_1,Y_2,Y_3\}$
    \ENDIF
    
    \STATE \textbf{$\{Y_1\}$-Update:} 
    Fix $Y_2$, $Y_3$, and $\mathsf{\Lambda}$. Minimize $\nu$ 
    subject to the LMI \eqref{eq:lmi6}, with 
    $\{P,\alpha,\nu,Y_1\}$ being the decision variables.
    
    \IF{$\nu < 0$}
        \STATE \textbf{Return:} \textsc{Certified} and the corresponding 
        matrices $\{P,\alpha,\gamma,\mathsf{\Lambda},Y_1,Y_2,Y_3\}$
    \ENDIF
    
    \IF{$k \ge K_{\max}$}
        \STATE \textbf{break}
    \ENDIF
\ENDWHILE
\STATE \textbf{Return:} \textsc{Unknown}
\end{algorithmic} 
\end{algorithm}

The above multi-step update scheme can be viewed as one iteration of an
approximate block coordinate descent (ABCD) method, in which different blocks
of decision variables are updated successively.

This observation naturally motivates a multi-round alternating algorithm.
Initially, choose $Y_1\prec0$ and set $\nu=0$.
In each round, feasible $\mathsf{\Lambda}_{N-1}$ is first computed from the
network weight matrices and the current $Y_1$ using a scalable Lipschitz
constant estimation algorithm.
Then, with $\mathsf{\Lambda}_{N-1}$ fixed, the variables
$\{\Lambda_N,Y_2,Y_3\}$ are updated by solving~\eqref{eq:lmi4}, and $Y_1$ is
subsequently updated by solving~\eqref{eq:lmi6}.
The updated $Y_1$ obtained at the end of the current round is used to
initialize the next round.

To ensure finite termination, Algorithm~\ref{al:ABCD} is executed for at most
$K_{\max}$ rounds.
If either update step yields $\nu<0$, the algorithm terminates with status
\textsc{Certified}, and the corresponding matrices
$\{P,\alpha,\gamma,\mathsf{\Lambda},Y_1,Y_2,Y_3\}$ constitute a feasible
solution of the LMI conditions~\eqref{eq:key2} and~\eqref{eq:key3} in
Lemma~\ref{lemma2}.
If no such certificate is obtained within $K_{\max}$ rounds, the algorithm
terminates with status \textsc{Unknown}.
Importantly, an \textsc{Unknown} status does not imply infeasibility of the
original LMI conditions~\eqref{eq:key2} and \eqref{eq:key3}, since the alternating procedure
provides a sufficient search strategy rather than a complete feasibility test.
Algorithm~\ref{al:ABCD} summarizes the proposed multi-round ABCD procedure.

Moreover, for iterations beyond the first round, the $Y_1$-update step
continues to guarantee a non-increasing value of $\nu$, whereas the
$\{Y_2,Y_3\}$-update step does not admit such a monotonicity guarantee.
The reason is that the calculation of $\mathsf{\Lambda}_{N-1}$ in the
$\{Y_2,Y_3\}$-update step is performed using a scalable Lipschitz estimation
algorithm, which differs from directly minimizing $\nu$ while treating
$\mathsf{\Lambda}$ as a decision variable in the SDP.
Nevertheless, in the numerical experiments of Section~\ref{sec:num}, only a
few rounds of ABCD are typically required, and in some cases a single round
suffices.

%% file: Further_properties.tex
\section{Implications and Limitations of Theorem \ref{Theorem:largeSDP1}}\label{sec:sufficient_condition}

In this section, we further investigate the implications and limitations 
of the LMI condition in Theorem~\ref{Theorem:largeSDP1}. 
In particular, we show that, under additional assumptions,
feasibility of the LMI guarantees boundedness of individual
closed-loop state trajectories, while it does not in general imply
global asymptotic convergence to a fixed point.

\begin{theorem}[Boundedness of the state trajectory]\label{Theorem:bounded}
Under the assumptions of Theorem~\ref{Theorem:largeSDP1}, further suppose that
$\rho\in(0,1)$ and the external input satisfies $d=0$.
If the LMI \eqref{eq:robust-condition} is feasible, then, for any initial
condition $x_0\in\mathbb{R}^{n_x}$, the corresponding state trajectory satisfies $\sup_{k\in\mathbb{Z}_{\geq 0}}\|x_k\|_2 < \infty$.
\end{theorem}

The following theorem shows that the conditions of Theorem~\ref{Theorem:largeSDP1} are insufficient to guarantee asymptotic convergence of the state trajectory.

\begin{theorem}[Counterexample]\label{them:counterexample}
Under the assumptions of Theorem~\ref{Theorem:largeSDP1}, further suppose
that $\rho\in(0,1]$ and $d=0$.
For an uncertain perturbation $\Delta_{\operatorname{U}}$ satisfying Assumption~\ref{assumption1},
feasibility of the LMI~\eqref{eq:robust-condition} does not, in general,
imply pointwise asymptotic convergence of the closed-loop state trajectory.
In particular, there exists an admissible time-varying interconnection
satisfying all the above assumptions for which
\eqref{eq:robust-condition} is feasible, yet the corresponding closed-loop
state trajectories do not converge to any fixed point.
\end{theorem}

The following remark clarifies an additional condition under which global convergence can be recovered.

\begin{remark}
Suppose that the external input $d=0$ and there exists a trajectory $(\zeta,v,w,p,q,d,e)$ of the interconnection starting from an initial condition $x_0 \in \mathbb{R}^{n_x}$ such that $\lim_{k\to\infty} x_k = x^\star$
for some $x^\star \in \mathbb{R}^{n_x}$.
 If the conditions in Theorem~\ref{Theorem:largeSDP1} hold, then the incremental convergence result in Theorem~\ref{Theorem:largeSDP1} implies that every trajectory starting from any initial condition in $\mathbb{R}^{n_x}$ also converges to the same limit point $x^\star$.  
Therefore, under an additional assumption ensuring the existence of at least one convergent trajectory, the incremental convergence property established in Theorem~\ref{Theorem:largeSDP1} can be strengthened to global convergence to $x^\star$. 

Similar assumptions have appeared in the literature. For example,
\cite{yin2021stability} assumes that the closed-loop equilibrium is
located at the origin, which directly provides a convergent reference
trajectory.
\end{remark}

%% file: Numerical_modified.tex
\section{Numerical Experiments}
\label{sec:num}
\subsection{Synthetic Examples}

This section presents two synthetic instances to evaluate the
scalability and effectiveness of the proposed framework. In both
instances, the objective is to compute certified upper bounds on
the incremental $\ell_2$ gain.

\subsubsection{Setup}

\noindent\textbf{Instance 1 (Scalar Plant).}
The LTI plant $G$ is given by
\begin{align}
\begin{split}
\nonumber
x_{k+1}&= -0.5\,x_{k}+ \underbrace{\bmat{1&\dots&1}}_{n_w}\, w_k +
0.5\,q_k + 0.4\,d_k,\\
v_k &= x_k\\
p_{k}&= 2.5\,x_{k} + 0.6\,d_k,\\
e_{k}&= 2\,x_{k} + 0.9\,d_k.
\end{split}
\end{align}
The uncertain perturbation $\Delta_\text{U}$ is described by two incremental $\rho$-hard IQCs
$\{\Psi_i,M_i\}_{i=1}^2$ with $\rho = 1$.
The first IQC $\{\Psi_1,M_1\}$ is associated with   $M_1=\mathrm{diag}(1,-1)$,
where $\Psi_1$ is given by 
\begin{align}
\begin{split}
\nonumber
\psi_{k+1}&= -0.3\,\psi_k + 1.3\,p_k,\\
r_k^{(1)}&=
\begin{bmatrix}0\\ -0.1\end{bmatrix}\psi_k+
\begin{bmatrix}0.2 & 0\\ 0 & -0.1\end{bmatrix}
\begin{bmatrix}p_k\\ q_k\end{bmatrix}.
\end{split}
\end{align}
The second IQC $\{\Psi_2,M_2\}$ is associated with $M_2=\mathrm{diag}(1,-1)$,
where $\Psi_2$ is given by
\[
r_k^{(2)} =
\begin{bmatrix}
-0.5 & 0.3\\
0 & 1.7
\end{bmatrix}
\begin{bmatrix}p_k\\ q_k\end{bmatrix},
\]
which is static, i.e., it has no internal dynamics.

\noindent\textbf{Instance 2 (Five-dimensional Plant).}
To further demonstrate the proposed framework on higher-dimensional
systems, we consider the following LTI plant $G$
\begin{align}
\begin{split}
\nonumber
x_{k+1} &= A_G x_k + B_{G1} w_k + B_{G2}q_k + B_{G3} d_k \\ 
v_k &= x_k\\
p_k &= C_{G2} x_k + D_{G23}d_k \\
e_k &= C_{G3} x_k + D_{G33}d_k,
\end{split}
\end{align}
with system matrices given by
\begin{align}
\begin{split}
\nonumber
 &\scalebox{0.83}{$
A_G = \begin{bmatrix}
0.52 & 0.12 & -0.08 & 0.05 & 0.03 \\
0.12 & 0.47 & 0.11 & -0.06 & 0.04 \\
-0.08 & 0.11 & 0.38 & 0.09 & -0.07 \\
0.05 & -0.06 & 0.09 & 0.31 & 0.10 \\
0.03 & 0.04 & -0.07 & 0.10 & 0.26
\end{bmatrix},\ B_{G1} = \underbrace{\bmat{1 & \dots & 1\\1 & \dots & 1\\1 & \dots & 1\\1 & \dots & 1\\1 & \dots & 1}}_{n_w}$},\\
 &\scalebox{0.83}{$
 B_{G2} = \bmat{ 0.5\\
      0.2\\
       0.4\\
       0.6\\
      0.3},\ B_{G3} = \bmat{0.8\\
       0.3\\
      0.2\\
       0.6\\
       0.5},\ C_{G2} = \bmat{0.4& 0.4& 0.8& 0.3& 0.7},
 $}\\
  &\scalebox{0.83}{$
  D_{G23} = 0.5,\ C_{G3} = \bmat{0.3 & 0.9 & 0.6 & 0.4 & 0.7},\ D_{G33} = 0.5.
  $}
\end{split}
\end{align}
The uncertain perturbation $\Delta_\text{U}$ is described by the same two incremental 
$\rho$-hard IQCs $\{\Psi_i,M_i\}_{i=1}^2$ as in Instance~1.

We evaluate the ABCD framework combined with several scalable Lipschitz constant estimation algorithms.  
Specifically, we consider ABCD coupled with the following six methods: ECLipsE, ECLipsE-SN, ECLipsE-GC, ECLipsE-GCS, ECLipsE-Shift, and EP-LipSDP. We do not include ECLipsE-Fast, since it is a special case of ECLipsE-SN \cite{syed2025improved}.  
For ECLipsE-SN, ECLipsE-GC, ECLipsE-GCS, and ECLipsE-Shift, we follow the setup in~\cite{syed2025improved}.   
For EP-LipSDP, we solve it using the LipDiff algorithm proposed in \cite{wang2024scalability}.

We consider three baselines. 
The first is the full-order SDP in Theorem~\ref{Theorem:largeSDP1}, which serves as
the least conservative baseline  but does not scale to deep networks.
The second is the small-gain condition in Proposition \ref{proposition2} combined with scalable Lipschitz estimation methods, which scales well but is more conservative. 
The third is the Chordal DeepSDP method~\cite{xue2024chordal}, which exploits chordal sparsity to decompose the LMI and improves computational efficiency, while its memory and computational costs still grow with network depth.

To assess scalability, we vary the depth of the NN component. Tables~\ref{tb:instance1} and~\ref{tb:instance2} report the incremental $\ell_2$-gain upper bounds for Instances~1 and~2, respectively, using networks with $\{20,40,60,80,100,5000\}$ layers and $50$ neurons per layer ($n_w=50$). Fig.~\ref{fig:computation_time} reports the corresponding computational time for both instances using networks with $20$ neurons per layer ($n_w=20$) and $\{20,40,60,80,100,120\}$ layers.

For each method, we progressively decrease $\sqrt{\gamma}$ and repeatedly solve
the corresponding verification problem to identify the smallest certified
value of $\sqrt{\gamma}$ found. The search resolution for $\sqrt{\gamma}$ is set to
$0.01$. For the proposed ABCD framework, we set the maximum number of rounds
to $K_{\max}=10$, and a trial value of $\sqrt{\gamma}$ for which
Algorithm~\ref{al:ABCD} returns \textsc{Unknown} is not interpreted as
infeasible. The incremental $\ell_2$-gain upper bounds
reported in Tables~\ref{tb:instance1} and \ref{tb:instance2} are the resulting certified values of
$\sqrt{\gamma}$. The reported computation time measures the runtime of a single verification
run at the final reported value of $\sqrt{\gamma}$. The time spent searching
over different candidate values of $\sqrt{\gamma}$ is not included.

All methods are implemented in MATLAB using CVX with MOSEK as the underlying solver. All experiments are conducted on a standard laptop equipped with $16$~GB of RAM. In the tables, ``$\sim$'' indicates solver failure due to insufficient memory, and ``$>1$h'' indicates that the computation did not finish within one hour. Entries highlighted in red denote the tightest upper bound among all methods, whereas entries highlighted in blue denote the best result within the ABCD framework. 
\vspace{0.5em}
\subsubsection{Results}

From Tables~\ref{tb:instance1}--\ref{tb:instance2}, when the network is small, the full-order SDP in Theorem~\ref{Theorem:largeSDP1} and Chordal-DeepSDP achieve the tightest upper bounds because they fully exploit the coupled SDP structure. However, the proposed ABCD framework combined with EP-LipSDP achieves upper bounds comparable to those achieved by these baselines.

As the network size increases, Theorem~1 and Chordal-DeepSDP
can no longer be solved due to memory limitations, while
ABCD + EP-LipSDP continues to produce the tightest upper bounds
among all tractable methods.
For extremely deep networks (e.g., 5000 layers), EP-LipSDP can no longer be solved due to memory limitations and ECLipsE requires more than one hour to complete. 
In this regime, ABCD combined with ECLipsE-GC or ECLipsE-GCS provides the tightest upper bounds among all tractable approaches.

Fig.~\ref{fig:computation_time} reports the computational time of the proposed methods and the baselines. To avoid cluttering the figure, we include only ECLipsE-GCS as a representative of the ECLipsE-Fast variants, since the other ECLipsE-Fast variants exhibit similar computational trends. As shown in Fig.~\ref{fig:computation_time}, ABCD-based methods have runtimes comparable
to those of small-gain approaches while providing significantly
tighter upper bounds on the incremental $\ell_2$ gain.

Overall, these results demonstrate that the proposed ABCD framework is less conservative than the incremental small-gain method while enabling a tunable tradeoff between computational and memory efficiency and conservatism through the choice of scalable Lipschitz constant estimation algorithms.

\begin{table}[htp]
\centering
\caption{Upper bounds on the incremental $\ell_2$ gain versus network depth (Instance 1)}
\label{tb:instance1}
\resizebox{\linewidth}{!}{
\begin{tabular}{c|cccccc}
\toprule
\textbf{Method} & 20 & 40 & 60 & 80 & 100 & 5000 \\
\midrule
\multicolumn{7}{c}{\textbf{Baselines}}\\
\midrule
Theorem~\ref{Theorem:largeSDP1} & \textcolor{red}{6.48} & \textcolor{red}{6.47} & \textcolor{red}{6.41} & $\sim$ & $\sim$ & $\sim$ \\
Chordal-DeepSDP & \textcolor{red}{6.48} & \textcolor{red}{6.47} & \textcolor{red}{6.41} & $\sim$ & $\sim$ & $\sim$ \\
Small-Gain + ECLipsE & 7.86 & 7.89 & 7.39 & 7.56 & 7.67 & $>$1h \\
Small-Gain + ECLipsE-SN & 10.75 & 10.23 & 9.16 & 10.40 & 14.74 & 10.61 \\
Small-Gain + ECLipsE-GC & 9.28 & 8.80 & 8.11 & 8.46 & 9.47 & 7.61 \\
Small-Gain + ECLipsE-GCS & 9.38 & 8.91 & 8.19 & 8.61 & 9.86 & 7.61 \\
Small-Gain + ECLipsE-Shift & 10.44 & 10.01 & 9.05 & 10.23 & 14.42 & 9.53 \\
Small-Gain + EP-LipSDP & 7.40 & 7.44 & 7.14 & 7.28 & 7.35 & $\sim$ \\

\midrule
\multicolumn{7}{c}{\textbf{ABCD Framework}}\\
\midrule
ABCD + ECLipsE & 6.88 & 6.90 & 6.55 & 6.67 & 6.74 & $>$1h \\
ABCD + ECLipsE-SN & 8.60 & 8.33 & 7.71 & 8.41 & 10.32 & 8.52 \\
ABCD + ECLipsE-GC & 7.75 & 7.48 & 7.03 & 7.25 & 7.88 & \textcolor{red}{6.92} \\
ABCD + ECLipsE-GCS & 7.83 & 7.57 & 7.11 & 7.37 & 8.12 & \textcolor{red}{6.92} \\
ABCD + ECLipsE-Shift & 8.48 & 8.26 & 7.67 & 8.36 & 10.28 & 7.98 \\
ABCD + EP-LipSDP & \textcolor{blue}{6.51} & \textcolor{blue}{6.52} & \textcolor{blue}{6.46} & \textcolor{red}{6.49} & \textcolor{red}{6.49} & $\sim$ \\

\bottomrule
\end{tabular}}
\end{table}

\begin{table}[htp]
\centering
\caption{Upper bounds on the incremental $\ell_2$ gain versus network depth (Instance 2)}
\label{tb:instance2}
\resizebox{\linewidth}{!}{
\begin{tabular}{c|cccccc}
\toprule
\textbf{Method} & 20 & 40 & 60 & 80 & 100 & 5000 \\
\midrule
\multicolumn{7}{c}{\textbf{Baselines}}\\
\midrule
Theorem~\ref{Theorem:largeSDP1} & \textcolor{red}{7.76} & \textcolor{red}{7.80} & \textcolor{red}{7.68} & $\sim$ & $\sim$ & $\sim$ \\
Chordal-DeepSDP & \textcolor{red}{7.76} & \textcolor{red}{7.80} & \textcolor{red}{7.68} & $\sim$ & $\sim$ & $\sim$ \\
Small-Gain + ECLipsE & 10.64 & 10.68 & 10.58 & 10.59 & 10.81 & $>$1h \\
Small-Gain + ECLipsE-SN & 11.60 & 11.65 & 11.57 & 11.53 & 12.45 & 11.48 \\
Small-Gain + ECLipsE-GC & 10.74 & 10.84 & 10.73 & 10.73 & 11.40 & 10.29 \\
Small-Gain + ECLipsE-GCS & 10.75 & 10.87 & 10.76 & 10.78 & 11.50 & 10.29 \\
Small-Gain + ECLipsE-Shift & 10.91 & 10.95 & 10.87 & 10.85 & 11.74 & 11.12 \\
Small-Gain + EP-LipSDP & 9.82 & 10.26 & 9.20 & 9.35 & 9.72 & $\sim$ \\

\midrule
\multicolumn{7}{c}{\textbf{ABCD Framework}}\\
\midrule
ABCD + ECLipsE & 8.25 & 8.49 & 7.90 & 7.94 & 8.18 & $>$1h \\
ABCD + ECLipsE-SN & 9.23 & 9.50 & 9.21 & 8.98 & 9.20 & 8.91 \\
ABCD + ECLipsE-GC & 8.85 & 9.16 & 8.54 & 8.50 & 8.75 & {\color{red}7.72} \\
ABCD + ECLipsE-GCS & 8.88 & 9.26 & 8.62 & 8.59 & 8.89 & 7.73 \\
ABCD + ECLipsE-Shift & 9.03 & 9.28 & 9.01 & 8.86 & 9.09 & 8.76 \\
ABCD + EP-LipSDP & \textcolor{blue}{7.91} & \textcolor{blue}{8.01} & \textcolor{blue}{7.75} & \textcolor{red}{7.73} & \textcolor{red}{7.98} & $\sim$ \\
\bottomrule
\end{tabular}}
\end{table}

\begin{figure}[t]
\centering

\begin{subfigure}{1\linewidth}
\centering
\includegraphics[width=\linewidth]{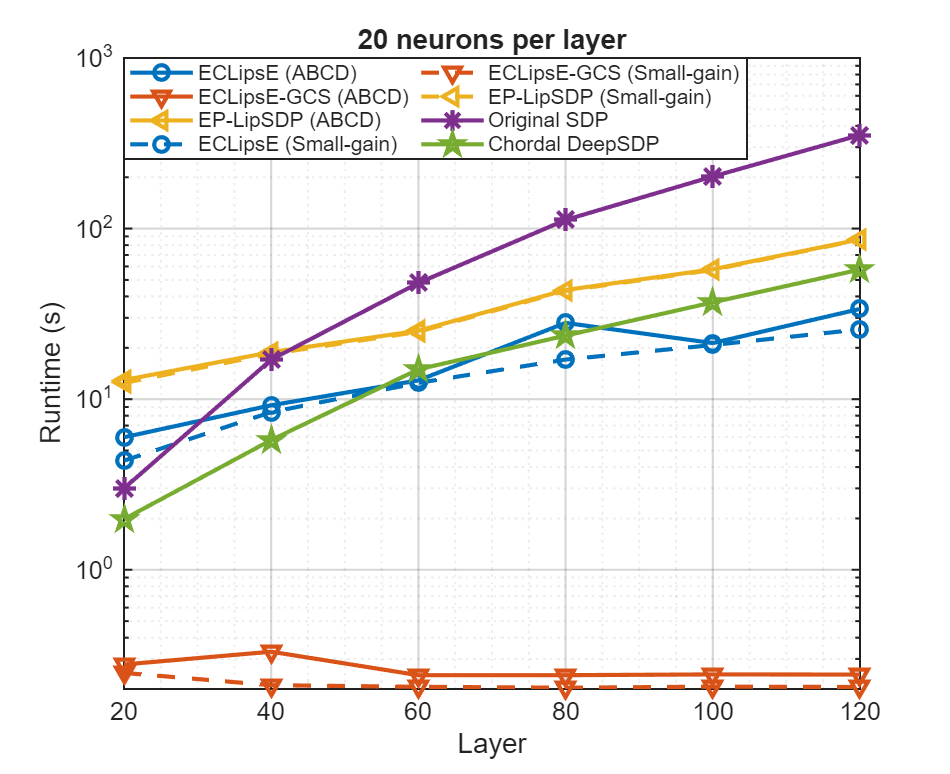}
\caption{Instance 1}
\end{subfigure}
\hfill
\begin{subfigure}{1\linewidth}
\centering
\includegraphics[width=\linewidth]{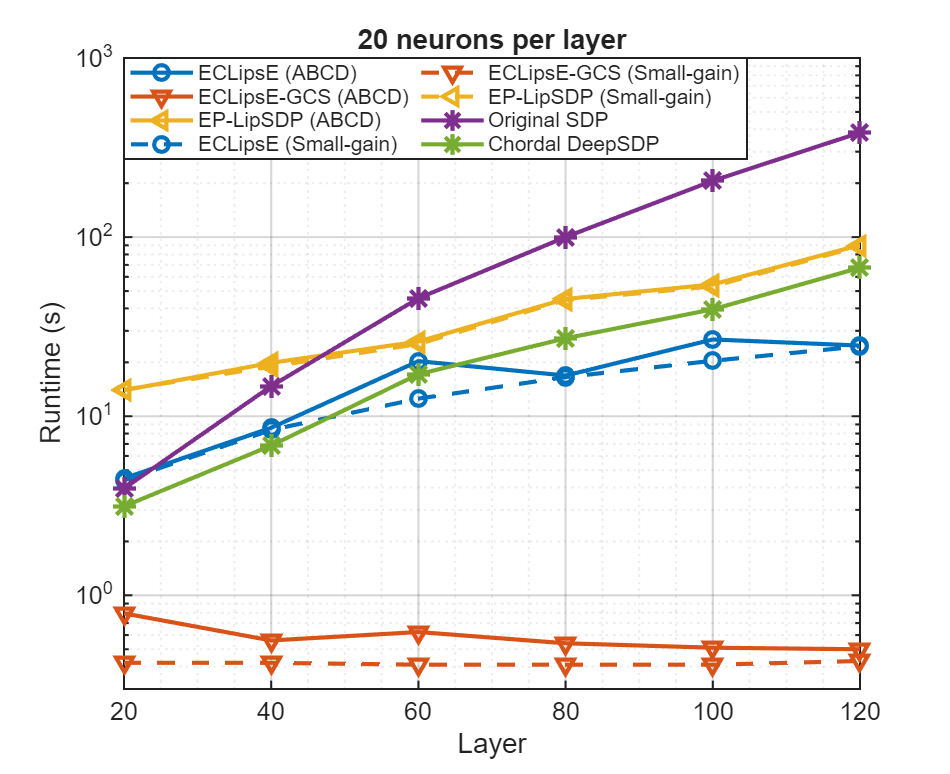}
\caption{Instance 2}
\end{subfigure}

\caption{Computational time versus network depth.}
\label{fig:computation_time}
\end{figure}

\subsection{Docking Case Study}
\label{subsec:docking}

We next consider a spacecraft docking example adapted from the
ARCH-COMP AINNCS benchmark \cite{ARCH25:ARCH_COMP25_Category_Report} to demonstrate the applicability
of the proposed framework to a physically meaningful closed-loop
control problem. We adopt the relative-motion dynamics and physical parameters from the benchmark, while constructing a new NN-augmented closed-loop system for the verification study.

\vspace{0.5em}

\subsubsection{System Description}

Consider the two-dimensional spacecraft docking problem based on the
Clohessy--Wiltshire relative-motion model \cite{ARCH25:ARCH_COMP25_Category_Report}. The state is defined as
$x_k =
    \begin{bmatrix}
        s_{x,k} &
        s_{y,k} &
        \dot{s}_{x,k} &
        \dot{s}_{y,k}
    \end{bmatrix}^{\tp}$,
where $s_{x,k}$ and $s_{y,k}$ denote the relative position of the
active deputy spacecraft with respect to the passive chief spacecraft
in Hill's reference frame, and $\dot{s}_{x,k}$ and $\dot{s}_{y,k}$
denote the corresponding relative velocities.

Following the benchmark setup, we use the spacecraft mass $m_{\rm sc}=12$ kg,
orbital mean motion $n_{\rm orb}=0.001027$ rad/s, and sampling period $T_s=1$ s.
The discretized dynamics are
\begin{equation}
\nonumber
    x_{k+1}
    =
    A_d x_k
    +
    B_d u_k,
\end{equation}
where
\begin{equation}
\nonumber
A_d
=
\begin{bmatrix}
1 & 0 & 1 & 0\\
0 & 1 & 0 & 1\\
3n_{\rm orb}^2 & 0 & 1 & 2n_{\rm orb}\\
0 & 0 & -2n_{\rm orb} & 1
\end{bmatrix},
\qquad
B_d
=
\begin{bmatrix}
0 & 0\\
0 & 0\\
1/m_{\rm sc} & 0\\
0 & 1/m_{\rm sc}
\end{bmatrix}.
\end{equation}
The input $u_k=
    \begin{bmatrix}
        u_{x,k} & u_{y,k}
    \end{bmatrix}^{\tp}$
represents the control forces along the two axes of Hill's reference
frame.

For the closed-loop verification study, we design a nominal LQR
controller $u_k^{\mathrm{LQR}}=K_{\mathrm{LQR}}x_k$
and augment it with a neural-network residual. The resulting commanded
control force is
\begin{equation}
\nonumber
    p_k
    =
    K_{\mathrm{LQR}}x_k
    +
    w_k,
\end{equation}
where $w_k=\Delta_{\mathrm{NN}}(x_k)$ denotes the NN output. The NN is
trained by supervised learning to approximate a bounded nonlinear
correction over the operating region of interest. The NN has
dimensions $4$-$256$-$256$-$2$ and uses the $\tanh$ activation function
at each layer, with a bias term included in each layer.

To account for actuator uncertainty, we consider independent
multiplicative gain variations in the two thrust channels,
\begin{equation}
    \nonumber q_k=\Delta_{\mathrm{U}}(p_k), \qquad
    \Delta_{\mathrm{U}}
    =\operatorname{diag}(\delta_1,\delta_2), \qquad
    |\delta_i|\leq\epsilon_{\mathrm{U}},
\end{equation}
where $\epsilon_{\mathrm{U}}=0.05$. The uncertainty is described by
the static incremental $\rho$-hard IQC $\{\Psi,M\}$ with $\rho=1$,
where
\begin{equation}
\nonumber
    r_k=
    \begin{bmatrix}
        \epsilon_{\mathrm{U}}I_2 & 0\\
        0 & I_2
    \end{bmatrix}
    \begin{bmatrix}
        p_k\\ q_k
    \end{bmatrix},
    \qquad
    M=\operatorname{diag}(I_2,-I_2).
\end{equation}
Including an additive thrust disturbance $d_k\in\mathbb{R}^2$, the
actual plant input is $u_k=p_k+q_k+d_k$. Hence, the closed-loop
dynamics are
\begin{equation}
\nonumber
    x_{k+1}
    =
    (A_d+B_dK_{\mathrm{LQR}})x_k
    +B_dw_k+B_dq_k+B_dd_k .
\end{equation}
This is a special case of the general interconnection considered in
Section~\ref{sec:preliminiries}. In particular,
\begin{align}
\begin{split}
\nonumber
    & A_G = A_d+B_dK_{\mathrm{LQR}},\ 
    B_{G1}=B_d,\
    B_{G2}=B_d, \
    B_{G3}=B_d,\\
    & C_{G1}=I_4,\ 
    D_{G12}=D_{G13}=0,\\
   &C_{G2}=K_{\mathrm{LQR}},\
    D_{G21}=I_2,\
    D_{G22}=D_{G23}=0.
    \end{split}
\end{align}

Finally, we choose the performance output as
$e_k=x_k$,
so that
\begin{equation}
\nonumber
    C_{G3}=I_4,\
    D_{G31}=D_{G32}=D_{G33}=0.
\end{equation}
Hence, the incremental $\ell_2$ gain from the additive thrust disturbance $d$ to the relative-motion state $e$ quantifies the closed-loop disturbance attenuation performance.

\vspace{0.5em}

\subsubsection{Verification Results}

We evaluate the proposed verification methods on the closed-loop
docking system described above. The experimental setup and solver
settings are the same as those used in the previous numerical
experiments. Table~\ref{tab:docking_results} compares the resulting incremental $\ell_2$ gain bounds with the full-order SDP in Theorem~\ref{Theorem:largeSDP1}, Chordal-DeepSDP, and the small gain approach. Theorem~1 and Chordal-DeepSDP provide the tightest gain bound of
$1.15$. Among the scalable approaches, ABCD consistently yields
tighter bounds than the corresponding small gain methods. In particular, ABCD with EP-LipSDP achieves
a gain bound of $1.23$, which is close to the full-order SDP bound, while
the corresponding small-gain bound is $1.79$.

\begin{table}[t]
    \centering
    \caption{Upper bound of the incremental $\ell_2$ gain for the
    spacecraft docking example.}
    \label{tab:docking_results}
    \begin{tabular}{l|c}
        \hline
        \textbf{Method} & \textbf{$\ell_2$ Gain Bound} \\
        \hline
        \multicolumn{2}{c}{\textbf{Baselines}} \\
        \hline
        Theorem \ref{Theorem:largeSDP1}                    & \textcolor{red}{1.15} \\
        Chordal-DeepSDP              & \textcolor{red}{1.15} \\
        Small-Gain + ECLipsE         & 1.96 \\
        Small-Gain + ECLipsE-SN      & 2.10 \\
        Small-Gain + ECLipsE-GC      & 2.08 \\
        Small-Gain + ECLipsE-GCS     & 2.08 \\
        Small-Gain + ECLipsE-Shift   & 2.01 \\
        Small-Gain + EP-LipSDP       & 1.79 \\
        \hline
        \multicolumn{2}{c}{\textbf{ABCD Framework}} \\
        \hline
        ABCD + ECLipsE               & 1.38 \\
        ABCD + ECLipsE-SN            & 1.48 \\
        ABCD + ECLipsE-GC            & 1.46 \\
        ABCD + ECLipsE-GCS           & 1.46 \\
        ABCD + ECLipsE-Shift         & 1.45 \\
        ABCD + EP-LipSDP             & \textcolor{blue}{1.23} \\
        \hline
    \end{tabular}
\end{table}

%% file: Appendix_modified.tex
\appendices

\section{Interpretation of the LMI Condition \eqref{eq:robust-condition}}\label{whyLMI}

We provide a brief derivation to clarify the structure of the LMI
condition in \eqref{eq:robust-condition}. Define the stacked vector of
all layer activations as
$\mathbf{h}
:=
\begin{bmatrix}
(h^1)^\tp&\cdots&(h^N)^\tp
\end{bmatrix}^\tp$.
Then, the augmented NN in \eqref{eq:Delta2} can be equivalently written as
\begin{equation}
    \mathbf{h}
    =
    \phi\left(
    \mathbf{W}\mathbf{h}
    +\mathbf{U}\xi+\mathbf{b}
    \right),
    \label{eq:stacked_nn}
\end{equation}
where 
\[
\mathbf{W}
=
\begin{bmatrix}
0      &        &        & 0\\
W_2    & 0      &        &  \\
       & \ddots & \ddots &  \\
0      &        & W_N    & 0
\end{bmatrix}.
\qquad
\mathbf{U}
=
\begin{bmatrix}
\widetilde W_1\\
0\\
\vdots\\
0
\end{bmatrix},
\qquad
\mathbf{b}
=
\begin{bmatrix}
b_1\\
b_2\\
\vdots\\
b_N
\end{bmatrix}.
\]
Thus, the recursive representation in \eqref{eq:Delta2} can be viewed
as a structured interconnection of $N$ slope-restricted nonlinearities
$\phi$, with $\mathbf{W}$ encoding the feedforward connections between
consecutive layers and $\mathbf{U}$ injecting the NN input into the
first layer.

Consider two inputs $\xi,\xi'$ and their corresponding stacked layer
activations $\mathbf{h},\mathbf{h}'$. Define
$\delta\xi:=\xi-\xi'$ and
$\delta\mathbf{h}:=\mathbf{h}-\mathbf{h}'$.
Since $\phi$ is slope-restricted in $[0,1]$, we have
\begin{equation}
    2\delta\mathbf{h}^{\top}\Lambda_{\rm blk}
    \left(
        \mathbf{W}\delta\mathbf{h}
        +\mathbf{U}\delta\xi
        -\delta\mathbf{h}
    \right)
    \geq 0,
    \label{eq:stacked_nn_qc}
\end{equation}
for any
$\Lambda_{\rm blk}=\operatorname{diag}(\Lambda_1,\ldots,\Lambda_N)$,
where each $\Lambda_l\succeq0$ is diagonal.
Expanding \eqref{eq:stacked_nn_qc} in terms of
$\begin{bmatrix}
\delta\xi^\tp,\delta\mathbf{h}^\tp
\end{bmatrix}^\tp$
gives
\[
\begin{bmatrix}
\delta\xi\\
\delta\mathbf{h}
\end{bmatrix}^{\tp}
\widetilde{\mathsf{M}}_{\mathrm{NN}}(\mathsf\Lambda)
\begin{bmatrix}
\delta\xi\\
\delta\mathbf{h}
\end{bmatrix}
\geq0,
\]
where $\widetilde{\mathsf{M}}_{\mathrm{NN}}(\mathsf\Lambda)$ is exactly the
matrix defined in \eqref{tildeNN}.

Finally, the remaining interconnected system interacts with the NN
only through its input $\xi$ and output $w=h^N$. Hence, its dissipation
inequality can be expressed in the same stacked coordinates by padding
the blocks corresponding to the intermediate layer activations with
zeros, yielding $\mathsf{M}_G(P,\gamma,\alpha)$ in \eqref{ABC}.
Combining this dissipation inequality with the NN quadratic constraint
above yields \eqref{eq:robust-condition}.

\section{Omitted Proofs for Section \ref{sec:background}}

\subsection{Proof of Theorem \ref{Theorem:largeSDP1}}

\begin{proof}
We prove the two claims separately.

\noindent\textbf{Incremental convergence.}
If LMI~\eqref{eq:robust-condition} holds, then there exists a scalar $\epsilon \in (0,1)$ such that
\begin{align}\label{eq:extendepsilon}
\begin{split}
&\widetilde{\mathsf{M}}_{\operatorname{NN}}(\mathsf{\Lambda})
+\mM_G(P, \gamma, \alpha) + \epsilon \operatorname{diag}\!\big(P,0_{n_\phi + n_q + n_d}\big)
\prec 0 .
\end{split}
\end{align}

Now consider two arbitrary trajectories
$(\zeta,v,w,p,q,d,e)$ and
$(\zeta^\prime,v^\prime,w^\prime,p^\prime,q^\prime,d^\prime,e^\prime)$
associated with initial conditions
$x_0,x_0^\prime\in\mathbb R^{n_x}$, where
$d_k=d_k^\prime=0$ for all $k\in \mathbb Z_{\ge 0 }$.
Define the incremental signals
$
\delta \zeta_k := \zeta_k-\zeta_k^\prime
$,
$
\delta v_k := v_k-v_k^\prime
$,
$
\delta \bold h_k := \bold h_k-\bold h_k^\prime
$, $
\delta q_k := q_k-q_k^\prime
$
and $
\delta d_k := d_k-d_k^\prime
$,
where
$
\bold h_k :=
\bmat{(h_k^1)^\tp & \cdots &(h_k^N)^\tp}^\tp
$ and $\bold h_k'$ defined analogously.
By the definition of $v_k$, we have $W_1 \delta v_k
=
\widetilde{W}_1\begin{bmatrix}
\delta\zeta_k^\tp &
\delta q_k^\tp &
\delta d_k^\tp
\end{bmatrix}^\tp$.
Hence,
\begin{align}
\begin{split}
&
\begin{bmatrix}
\delta\zeta_k^\tp &
\delta q_k^\tp &
\delta d_k^\tp &
\delta\bold h_k^\tp
\end{bmatrix}
\widetilde{\mathsf M}_{\mathrm{NN}}(\mathsf\Lambda)
\begin{bmatrix}
\star
\end{bmatrix}^\tp
\\
=&\;
\begin{bmatrix}
\delta v_k \\
\delta\bold h_k
\end{bmatrix}^\tp
\mathsf M_{\mathrm{NN}}(\mathsf\Lambda)
\begin{bmatrix}
\delta v_k \\
\delta\bold h_k
\end{bmatrix}
\ge0,
\end{split}
\end{align}
where $[\star]^\tp$ denotes the transpose of the preceding block vector. Therefore, premultiplying and postmultiplying~\eqref{eq:extendepsilon}
by $\begin{bmatrix}
\delta\zeta_k^\tp &
\delta q_k^\tp &
\delta d_k^\tp &
\delta\bold h_k^\tp
\end{bmatrix}$
and its transpose yields
\begin{align}
\begin{split}
\nonumber
V(\delta \zeta_{k+1})
- (1-\epsilon)V(\delta \zeta_k)
+ \delta r_k^\tp M(\alpha)\, \delta r_k  \le 0,\quad \forall k\in\mathbb{Z}_{\ge 0},
\end{split}
\end{align}
where $V(\delta\zeta) := \delta\zeta^\tp P \delta\zeta$ and $\delta r_k := r_k - r_k^\prime$.
Summing the above inequality from $k=0$ to $K$ for any $K\in\mathbb{Z}_{\ge 0}$ gives
\begin{align}
\begin{split}
\nonumber
V(\delta \zeta_{K+1})
+ \sum_{k=0}^K \epsilon\,V(\delta \zeta_k)
+ \sum_{k=0}^K \delta r_k^\tp M(\alpha)\, \delta r_k
\le V(\delta \zeta_0).
\end{split}
\end{align}
Since $\Delta_\text{U}$ satisfies incremental $\rho$-hard IQCs and, by Remark~\ref{remarkrho}, the same inequality holds with $\rho = 1$, the term $\sum_{k=0}^K \delta r_k^\tp M(\alpha)\, \delta r_k$
is nonnegative for all $K \in \mathbb{Z}_{\ge0}$. Therefore,
\begin{align}
\nonumber
V(\delta \zeta_{K+1})
+ \sum_{k=1}^K \epsilon\,V(\delta \zeta_k)
\le V(\delta \zeta_0).
\end{align}
Because $V(\cdot)\ge 0$ and $\epsilon>0$, the above inequality implies that
$\sum_{k=1}^\infty V(\delta \zeta_k) < \infty$,
and hence $V(\delta \zeta_k)\to 0$ as $k\to\infty$. Since $P\succ 0$, the state components satisfy
$\|x_k - x_k'\|_2 \to 0$.

\vspace{1em}

\noindent\textbf{Incremental $\ell_2$ gain upper bound.}
Now consider two arbitrary trajectories
$(\zeta,v,w,p,q,d,e)$ and
$(\zeta^\prime,v^\prime,w^\prime,p^\prime,q^\prime,d^\prime,e^\prime)$
of the above interconnection starting from
$x_0,x_0^\prime\in\mathbb R^{n_x}$, where
$d,d^\prime\in \ell_2^{n_d}$.
Premultiplying and postmultiplying \eqref{eq:robust-condition} by $\begin{bmatrix}
\delta\zeta_k^\tp &  \delta q_k^\tp& \delta d_k^\tp& \delta\bold{h}_k^\tp
\end{bmatrix}$ and its transpose yields
\begin{align}\label{term1&term2}
\begin{split}
&\norm{\delta e_k}_2^2
- \gamma \norm{\delta d_k}_2^2
+ V(\delta \zeta_{k+1})
- V(\delta \zeta_k) \\& +
\delta r_k^\tp M(\alpha) \delta r_k  \le 0, \quad k \in \mathbb{Z}_{\ge 0 },
\end{split}
\end{align}
where $\delta e_k := e_k - e_k^\prime $.
Since $\sum_{k=0}^K \delta r_k^\tp M(\alpha) \delta r_k \ge 0$ for all $K\in\mathbb{Z}_{\ge 0}$, summing~\eqref{term1&term2} from $k=0$ to any $K\in\mathbb{Z}_{\ge 0}$ yields
\[
\sum_{k=0}^K \norm{\delta e_k}_2^2
\le
\gamma \sum_{k=0}^K \norm{\delta d_k}_2^2
+ V(\delta \zeta_0).
\]
Based on Definition \ref{def:incrementalL2gain}, this shows that $\sqrt{\gamma}$ is a valid upper bound on the incremental $\ell_2$ gain.

\end{proof}

\subsection{Proof of Proposition \ref{proposition2}}

\begin{proof}  
We prove the two claims separately.

\noindent\textbf{Incremental convergence.}
Suppose that LMI~\eqref{eq:smallgain} holds. 
Then there exists a scalar $\epsilon \in (0,1)$ such that
\begin{align}\label{eq:extendepsilonsmallgain}
\begin{split}
&\bmat{\mathbf{A}(P,\gamma,\alpha) &\mathbf{B}(P,\alpha)\\
\mathbf{B}(P,\alpha)^\tp &\mathbf{C}(P,\alpha)} 
+ \text{diag}(I_{m},-\tfrac{1}{L}I_{n_w})
\\&+ \epsilon \operatorname{diag}\big(P,0_{n_w+n_q+n_d}\big) \prec 0.
\end{split}
\end{align}

Consider two arbitrary trajectories
$(\zeta,v,w,p,q,d,e)$ and
$(\zeta^\prime,v^\prime,w^\prime,p^\prime,q^\prime,d^\prime,e^\prime)$
associated with initial conditions
$x_0,x_0^\prime\in\mathbb R^{n_x}$, where
$d_k=d_k^\prime=0$ for all $k\in \mathbb Z_{\ge 0 }$.
Define
$\xi_k
:=
\bmat{\zeta_k^\tp&q_k^\tp&d_k^\tp}^\tp$,
and the incremental signals
$\delta w_k := w_k - w_k^\prime,
$ $\delta \xi_k := \xi_k - \xi_k^\prime$.
Premultiplying and postmultiplying
\eqref{eq:extendepsilonsmallgain}
by
$\begin{bmatrix}
\delta\zeta_k^\tp &
\delta q_k^\tp &
\delta d_k^\tp &
\delta w_k^\tp
\end{bmatrix}$
and its transpose yields
\begin{align}
\begin{split}
&
V(\delta \zeta_{k+1})
-
(1-\epsilon)V(\delta \zeta_k)
+
\delta r_k^\tp M(\alpha)\delta r_k
\\
&
+\|\delta \xi_k\|_2^2
-
\tfrac{1}{L}\|\delta w_k\|_2^2
\le
0,
\qquad
\forall k \in \mathbb{Z}_{\ge 0}.
\end{split}
\end{align}
 By assumption, $\sqrt{L}$ is an upper bound on the Lipschitz constant of 
$\Delta_{\mathrm{NN}}'$ as defined in \eqref{eq:Delta2}, so that $\|\delta w_k\|_2
\le \sqrt{L}\,\|\delta \xi_k\|_2$.
Consequently,
$\|\delta \xi_k\|_2^2 - \tfrac{1}{L}\|\delta w_k\|_2^2
\ge 0$,
and therefore the previous inequality implies
\[
V(\delta \zeta_{k+1})
- (1-\epsilon)V(\delta \zeta_k)
+ \delta r_k^\tp M(\alpha)\, \delta r_k
\le 0.
\]
The remainder of the argument proceeds exactly as in the proof of 
Theorem~\ref{Theorem:largeSDP1}.

\noindent\textbf{Incremental $\ell_2$ gain upper bound.}
Consider two arbitrary trajectories
$(\zeta,v,w,p,q,d,e)$ and
$(\zeta^\prime,v^\prime,w^\prime,p^\prime,q^\prime,d^\prime,e^\prime)$
of the above interconnection starting from $x_0,x_0^\prime\in\mathbb R^{n_x}$, where
$d,d^\prime\in \ell_2^{n_d}$. 
Premultiplying and postmultiplying \eqref{eq:smallgain} by 
$\begin{bmatrix}
\delta\zeta_k^\tp & \delta q_k^\tp & \delta d_k^\tp & \delta w_k^\tp
\end{bmatrix}$ 
and its transpose yields
\begin{align}
\begin{split}
&\norm{\delta e_k}_2^2
- \gamma \norm{\delta d_k}_2^2
+ V(\delta \zeta_{k+1})
- V(\delta \zeta_k)  +
\delta r_k^\tp M(\alpha) \delta r_k \\&
   + \|\delta \xi_k\|_2^2 - \tfrac{1}{L}\|\delta w_k\|_2^2
   \le 0,\quad \forall k \in \mathbb{Z}_{\ge 0}.
\end{split}
\end{align}
Since $\|\delta w_k\|_2\le\sqrt{L}\|\delta\xi_k\|_2$, the above inequality implies that
\begin{align}
\begin{split}
\nonumber
\norm{\delta e_k}_2^2
- \gamma \norm{\delta d_k}_2^2
+ V(\delta \zeta_{k+1})
- V(\delta \zeta_k)+
\delta r_k^\tp M(\alpha) \delta r_k 
   \le 0.
\end{split}
\end{align}
Summing the above inequality from $k=0$ to $K$ and proceeding as in the proof of 
Theorem~\ref{Theorem:largeSDP1} establishes the desired incremental $\ell_2$ gain bound.

\end{proof}

\section{Omitted Proofs for Section \ref{sec:normal}}

\subsection{Proof of Lemma \ref{lemma2}}

\begin{proof}
We first prove the ``if'' direction.
Suppose there exist matrices $\{Y_1,Y_2,Y_3\}$ such that
\eqref{eq:key2} and \eqref{eq:key3} hold.
From \eqref{eq:key2}, it follows that
\begin{align*}
    \mathsf{M}_G(P, \gamma,\alpha) - \mM(Y_1,Y_2,Y_3) \preceq 0.
\end{align*}
Adding this inequality to \eqref{eq:key3} yields
\eqref{eq:robust-condition}, which establishes the ``if'' direction.

We now prove the ``only if'' direction.
Suppose that \eqref{eq:robust-condition} holds for some
$\{P,\alpha,\gamma,\mathsf{\Lambda}\}$.
Since the inequality is strict, there exists $\epsilon > 0$ such that
\begin{align}
\begin{split}
\nonumber
&\widetilde{\mathsf{M}}_{\operatorname{NN}}(\mathsf{\Lambda})
    +\mathsf{M}_G(P, \gamma, \alpha) 
    +\epsilon\text{diag}\big(I_{m},0_{n_\phi-n_w},I_{n_w}\big)
    \prec 0.
\end{split}
\end{align}
Define $Y_1 := \mathbf{A}(P,\gamma,\alpha)+\epsilon I_{m}$, 
$Y_2 := \mathbf{C}(P,\alpha)+\epsilon I_{n_w}$, and
$Y_3 := \mathbf{B}(P,\alpha)$.
By direct substitution, this choice of $\{Y_1,Y_2,Y_3\}$ satisfies
both \eqref{eq:key2} and \eqref{eq:key3}, thereby completing the proof.
\end{proof}

\subsection{Proof of Theorem \ref{coro:main3}}

\begin{proof}
Suppose that $\{\mathsf\Lambda_{N-1},Y_1\}$ is given such that
$\widetilde{\mM}_{\mathrm{NN}}(\mathsf{\Lambda}_{N-1})+\operatorname{diag}(Y_1,0_{n_\phi-n_w})\prec 0$.
By the Schur complement, this condition implies that $\Xi_l \prec 0$ for all $l\in[N]$, where $\Xi_1=Y_1$, $\Xi_2$ is defined in \eqref{Xi2Gamma2} and $\Xi_l$, $l = 3,\dots,N$ is defined in \eqref{eq:ECLipsE-recursion}.

Applying the Schur complement, the matrix
$\widetilde{\mM}_{\mathrm{NN}}(\mathsf{\Lambda})+\mM(Y_1,Y_2,Y_3)$
is negative definite if and only if $\Xi_1\prec 0$, $\Xi_2\prec 0$, and
\begin{align}
\nonumber
\scalebox{0.85}{$
\bmat{
\Xi_2 & W_2^\tp \Lambda_2  & 0 & \cdots & \Gamma_2Y_3 \\
\Lambda_2 W_2 & -2\Lambda_2  & W_3^\tp \Lambda_3 & \ddots & \vdots \\
0 & \Lambda_3 W_3 & \ddots & \ddots & \vdots \\
\vdots & \ddots & \ddots & -2\Lambda_{N-1} & W_N^\tp \Lambda_N \\
Y_3^\tp\Gamma_2^\tp & \cdots & 0 & \Lambda_N W_N
& Y_2 - 2\Lambda_N - Y_3^\tp \Gamma_1^\tp \Xi_1^{-1} \Gamma_1 Y_3
}\prec 0.$}
\end{align}
By recursively applying the same argument, we obtain that \eqref{eq:key3} holds if and only if
$\Xi_l \prec 0$ for all $l\in[N]$ and
\begin{align}\label{theorem2eq1}
\scalebox{0.8}{$
\bmat{
\Xi_N & W_N^\tp \Lambda_N + \Gamma_N Y_3 \\
\Lambda_N W_N + Y_3^\tp \Gamma_N^\tp
& Y_2 - 2\Lambda_N - Y_3^\tp \big(\sum_{l=1}^{N-1} \Gamma_l^\tp \Xi_l^{-1} \Gamma_l\big) Y_3
}
\prec 0 .$}
\end{align}
Since $\{\mathsf\Lambda_{N-1},Y_1\}$ guarantees $\Xi_l\prec 0$ for all $l\in[N]$, feasibility of
\eqref{theorem2eq1} for some $\{Y_1,Y_2,Y_3\}$ immediately implies the validity of \eqref{eq:key3}.
Noting that \eqref{theorem2eq1} contains bilinear terms in $Y_3$, we rewrite it as
\begin{align}
\begin{split}
\nonumber
&\bmat{
\Xi_N & W_N^\tp \Lambda_N + \Gamma_N Y_3 \\
\Lambda_N W_N + Y_3^\tp \Gamma_N^\tp & Y_2 - 2\Lambda_N
}
\\&-
\bmat{0\\ Y_3^\tp}
\Big(\sum_{l=1}^{N-1}\Gamma_l^\tp \Xi_l^{-1}\Gamma_l\Big)
\bmat{0 & Y_3}
\prec 0 .
\end{split}
\end{align}
A further application of the Schur complement then yields
\begin{align}
\nonumber
\scalebox{0.85}{$
    \bmat{
        \Xi_N 
        & W_N^\tp\Lambda_N  + \Gamma_N Y_3 
        & 0\\
         \Lambda_NW_N + Y_3^\tp \Gamma_N^\tp 
        & Y_2 - 2\Lambda_N 
        & Y_3^\tp\\
        0 
        & Y_3 
        & \Big(\sum_{l=1}^{N-1} \Gamma_l^\tp \Xi_l^{-1} \Gamma_l \Big)^{-1}
    }
    \prec 0.$}
\end{align}
Note that $\sum_{l=1}^{N-1}\Gamma_l^\top \Xi_l^{-1}\Gamma_l$ is negative definite and hence invertible, since its first term equals $\Xi_1^{-1}$ with $\Xi_1 \prec 0$, while the remaining terms are negative semidefinite.
Therefore, if the two LMIs in~\eqref{key:lmi5} are satisfied, then the conditions
\eqref{eq:key2} and \eqref{eq:key3} hold for the same
$\{P,\alpha,\gamma,\mathsf{\Lambda},Y_1,Y_2,Y_3\}$. 
\end{proof}

\section{Omitted Proofs for Section \ref{sec:sufficient_condition}}

\subsection{Proof of Theorem \ref{Theorem:bounded}}

\begin{proof}
Define $c := \sum_{l=1}^{N} b_l^\tp b_l$.
If $c=0$, then all bias terms vanish. Since $\phi(0)=0$, $d=0$, and
$\Delta_{\operatorname{U}}(0)=0$, the zero trajectory is an admissible
reference trajectory. Hence, boundedness follows directly from the
incremental convergence result in Theorem~1. Therefore, it remains to
consider the case $c>0$.

Let $E := \begin{bmatrix}
C_2 & D_{22} & D_{23} &
0_{n_e\times(n_\phi-n_w)} & D_{21}
\end{bmatrix}$,
and define $\mM_G^{0}(P,\gamma,\alpha)$ as the matrix obtained from
$\mM_G(P,\gamma,\alpha)$ by removing the positive semidefinite
performance term $E^\tp E$, i.e.,
\[
    \mM_G(P,\gamma,\alpha)
    =
    \mM_G^{0}(P,\gamma,\alpha)+E^\tp E .
\]
Since $E^\tp E\succeq0$, feasibility of
\eqref{eq:robust-condition} implies
\[
    \widetilde{\mathsf{M}}_{\operatorname{NN}}(\mathsf{\Lambda})
    +\mM_G^{0}(P,\gamma,\alpha)
    \prec0.
\]
Moreover, $\mM_G^{0}(P,\gamma,\alpha)$ is homogeneous in
$(P,\gamma,\alpha)$.

Since the above inequality is strict, there exist scalars
$\epsilon>0$ and $0<\sigma\leq\frac{1-\rho}{c}$
such that
\begin{align}
\begin{split}
\nonumber
&\widetilde{\mathsf{M}}_{\operatorname{NN}}(\mathsf{\Lambda})
+\mM_G^{0}(P,\gamma,\alpha)
+\sigma\operatorname{diag}\!\big(
    cP,0_{n_\phi+n_q+n_d}
\big)
\\
&\quad
+\frac{\epsilon}{\sigma}
\operatorname{diag}\!\big(
    0_{m},
    \Lambda_1^2,\ldots,\Lambda_N^2
\big)
\prec0 .
\end{split}
\end{align}
Define $\tilde P:=\epsilon P$,
$\tilde\gamma:=\epsilon\gamma$,
$\tilde\alpha:=\epsilon\alpha$,
$\tilde\Lambda_l:=\epsilon\Lambda_l$,
and $\tilde{\mathsf{\Lambda}}
    :=\{\tilde\Lambda_l\}_{l=1}^N$.
Multiplying the preceding inequality by $\epsilon$ and using the
homogeneity of
$\widetilde{\mathsf{M}}_{\operatorname{NN}}$ and $\mM_G^{0}$ yields
\begin{align}\label{epsilonLMI}
\begin{split}
&\widetilde{\mathsf{M}}_{\operatorname{NN}}
    (\tilde{\mathsf{\Lambda}})
+\mM_G^{0}(\tilde P,\tilde\gamma,\tilde\alpha)
+\sigma\operatorname{diag}\!\big(
    c\tilde P,0_{n_\phi+n_q+n_d}
\big)
\\
&\quad
+\frac{1}{\sigma}
\operatorname{diag}\!\big(
    0_{m},
    \tilde\Lambda_1^2,\ldots,\tilde\Lambda_N^2
\big)
\prec0 .
\end{split}
\end{align}

Applying the Schur complement, \eqref{epsilonLMI} is equivalent to
\begin{align}\label{implicitLMI}
\begin{split}
&\operatorname{diag}\!\Big(
    \widetilde{\mathsf{M}}_{\operatorname{NN}}
        (\tilde{\mathsf{\Lambda}})
    +\mM_G^{0}(\tilde P,\tilde\gamma,\tilde\alpha),
    0_{n_\phi}
\Big)
\\
&\quad+
\begin{bmatrix}
\sigma\operatorname{diag}\!\big(
    c\tilde P,0_{n_\phi+n_q+n_d}
\big)
&
\mM_b(\tilde{\mathsf{\Lambda}})^\tp
\\
\mM_b(\tilde{\mathsf{\Lambda}})
&
-\sigma I_{n_\phi}
\end{bmatrix}
\prec0 ,
\end{split}
\end{align}
where
\[
\mM_b(\tilde{\mathsf{\Lambda}})
:=
\begin{bmatrix}
0 & \tilde\Lambda_1 & 0 & \cdots & 0\\
0 & 0 & \tilde\Lambda_2 & \cdots & 0\\
\vdots & \vdots & \vdots & \ddots & \vdots\\
0 & \cdots & \cdots & \cdots & \tilde\Lambda_N
\end{bmatrix}.
\]

Now consider an arbitrary trajectory
$(\zeta,v,w,p,q,d,e)$ associated with an initial condition
$x_0\in\mathbb{R}^{n_x}$ and satisfying
$d_k=0$ for all $k\in\mathbb{Z}_{\geq0}$. Since $\Delta_{\operatorname{U}}(0)=0$, the incremental IQC can be
applied to $(p,q)$ and the zero input-output pair, so that the
incremental filtered signal coincides with $r_k$.
Premultiplying and postmultiplying \eqref{implicitLMI} by $\begin{bmatrix}
    \zeta_k^\tp &
    q_k^\tp &
    d_k^\tp &
    \bold{h}_k^\tp &
    \bold{b}^\tp
\end{bmatrix}$
and its transpose, where
$\bold{b}
    :=
    \begin{bmatrix}
        b_1^\tp & \cdots & b_N^\tp
    \end{bmatrix}^\tp$,
and using the quadratic constraint associated with the
slope restriction of $\phi$ and $\phi(0)=0$, we obtain
\begin{align}
\begin{split}
\nonumber
&V(\zeta_{k+1})-V(\zeta_k)
+r_k^\tp M(\tilde\alpha)r_k
\\
&\quad
+\sigma\left(
    c-\sum_{l=1}^N b_l^\tp b_l
\right)
+\sigma c\big(V(\zeta_k)-1\big)
<0,
\qquad
\forall k\in\mathbb{Z}_{\geq0},
\end{split}
\end{align}
where $V(\zeta_k):=\zeta_k^\tp\tilde P\zeta_k$.
Using the definition of $c$, the constant term cancels. Define $\tilde V(\zeta_k):=V(\zeta_k)-1$.
Then
\begin{align}
    \tilde V(\zeta_{k+1})
    -(1-\sigma c)\tilde V(\zeta_k)
    +r_k^\tp M(\tilde\alpha)r_k
    <0 .
\end{align}
Recursively applying this inequality gives, for every
$K\in\mathbb{Z}_{\geq0}$,
\[
\tilde V(\zeta_{K+1})
<
(1-\sigma c)^{K+1}\tilde V(\zeta_0)
-
\sum_{k=0}^{K}
(1-\sigma c)^{K-k}
r_k^\tp M(\tilde\alpha)r_k .
\]

Since the uncertain perturbation satisfies the incremental $\rho$-hard IQCs and
$\rho\leq1-\sigma c$, Remark~\ref{remarkrho} gives
$\sum_{k=0}^{K}
    (1-\sigma c)^{-k}
    r_k^\tp M(\tilde\alpha)r_k
    \geq0$.
Multiplying both sides by $(1-\sigma c)^K$ yields $\sum_{k=0}^{K}
    (1-\sigma c)^{K-k}
    r_k^\tp M(\tilde\alpha)r_k
    \geq0$.
Hence,
\[
    V(\zeta_{K+1})
    <
    (1-\sigma c)^{K+1}
    \big(V(\zeta_0)-1\big)+1 .
\]

Finally, since
\[
    V(\zeta)
    =\zeta^\tp\tilde P\zeta
    \geq
    \lambda_{\min}(\tilde P)\|\zeta\|_2^2
\]
and $x_k$ is a component of $\zeta_k$, we obtain
\[
\|x_{K+1}\|_2
<
\sqrt{
\frac{
(1-\sigma c)^{K+1}
\big(V(\zeta_0)-1\big)+1
}{
\lambda_{\min}(\tilde P)
}},
\qquad
\forall K\in\mathbb{Z}_{\geq0}.
\]
Thus, $\sup_{k\in\mathbb{Z}_{\geq0}}\|x_k\|_2<\infty$.
\end{proof}

\subsection{Proof of Theorem \ref{them:counterexample}}

\begin{proof}
    Consider the following \emph{scalar} discrete-time system:
\begin{align*}
\nonumber
x_{k+1} &= 0.5 x_k + w_k + q_k,\\
v_k &= x_k\\
p_k &= x_k
\end{align*}
where $v_k,x_k,w_k,q_k \in \mathbb{R}$.
We define a special one layer NN:
\[
w_k = \phi(0\, v_k + b),\qquad \phi(x)=x,
\]
so that the output of the NN $w_k = b$ is a constant and can be seen as a constant offset. The uncertain perturbation $\Delta_\text{U}$ is given by
\[
q_k = \Delta_U(p_k) = \beta_k p_k,
\]
where the gain $\beta_k$ switches periodically with the parity of $k$:
\[
\beta_k=
\begin{cases}
0.1, & k\ \text{odd},\\
0.2, & k\ \text{even}.
\end{cases}
\]
It can be verified directly that $\Delta_U$ and the interconnection satisfy Assumption \ref{assumption1}.

For any two input sequences $p$ and $p^\prime$ (with the same switching rule $\{\beta_k\}$), the corresponding outputs satisfy
\[
\delta q_k := q_k-q_k^\prime = \beta_k (p_k- p_k^\prime)=\beta_k\,\delta p_k,\quad \forall k \in \mathbb{Z}_{\ge 0}.
\]
Since $\beta_k\in[0.1,0.2]$, the operator $\Delta_\text{U}$ satisfies the following incremental QC
\[
\begin{bmatrix}\delta p_k\\ \delta q_k\end{bmatrix}^{\!\tp}
\begin{bmatrix}
-0.02 & 0.15\\
0.15 & -1
\end{bmatrix}
\begin{bmatrix}\delta p_k\\ \delta q_k\end{bmatrix}\ge 0,\quad \forall k \in \mathbb{Z}_{\ge0}.
\]
Hence, the corresponding weighted sum is nonnegative for any
$\rho\in(0,1]$, and therefore $\Delta_U$ satisfies the required
incremental $\rho$-hard IQC.

For this instance, the LMI in Theorem \ref{Theorem:largeSDP1} reduces to
\begin{align}
\label{eq:lmi_instance}
\begin{split}
\nonumber
&\begin{bmatrix}
0.25P - P & 0.5P & 0.5P\\
0.5P & P & P\\
0.5P & P & P
\end{bmatrix}
+\begin{bmatrix}
0 & 0 & 0\\
0 & 0 & 0\\
0 & 0 & -2\Lambda
\end{bmatrix}\\&
+\alpha_1
\begin{bmatrix}
1 & 0\\
0 & 1\\
0 & 0
\end{bmatrix}
\begin{bmatrix}
-0.02 & 0.15\\
0.15 & -1
\end{bmatrix}
\begin{bmatrix}
1 & 0 & 0\\
0 & 1 & 0
\end{bmatrix}
\prec 0,
\end{split}
\end{align}
where $P>0$, $\Lambda \ge 0$, and $\alpha_1 \ge 0$ are decision variables. By solving the above LMI, we verify that it admits a feasible solution given by $P = 0.245$, $\Lambda = 0.533$ and $\alpha_1 = 1.418$.

However, the resulting closed-loop trajectories do not necessarily converge to any point. 
Indeed, for the above scalar instance,
\[
x_{k+1} = (0.5+\beta_k)x_k + b,
\]
where $\beta_k$ alternates between $0.1$ and $0.2$. If $b\neq 0$ and $\{x_k\}$ converges to some $x^\star$, then taking limits along even and odd subsequences yields
$x^\star = 0.6x^\star + b$ and $x^\star = 0.7x^\star + b$,
which imply $0.4x^\star=b$ and $0.3x^\star=b$, a contradiction. Hence, there exists an admissible interconnection for which the LMI is feasible, yet the closed-loop trajectories do not converge to any point.

\end{proof}